\documentclass[11pt]{article}
\usepackage{mathptmx}
\usepackage{geometry}                
\usepackage{graphicx}
\usepackage{epstopdf}
\usepackage{amsmath, amsthm, amssymb}

\newtheorem{lemma}{Lemma}

\newtheorem{theorem}[lemma]{Theorem}
\newtheorem{corollary}[lemma]{Corollary}

\theoremstyle{definition}
\newtheorem{example}[lemma]{Example}

\newcommand{\R}{{\bf R}}
\newcommand{\C}{{\bf C}}

\newcommand{\sign}{{\rm sign}}
\newcommand{\dist}{{\rm dist}}
\newcommand{\Spec}{{\rm Spec}}
\renewcommand{\Re}{{\rm Re}}
\renewcommand{\Im}{{\rm Im}}

\newcommand{\cC}{{\cal C}}

\newcommand{\cG}{{\cal G}}

\title{Spectral Properties of Matrices \\ Associated with Some Directed Graphs}
\author{E. B. Davies \ \ \ \  P. A. Incani}

\begin{document}
\maketitle

\begin{abstract}
We study the spectral properties of certain
non-self-adjoint matrices associated with large directed
graphs. Asymptotically the eigenvalues converge to certain
curves, apart from a finite number that have limits not on
these curves.
\end{abstract}

\section{Introduction}

In an earlier paper we studied the spectral properties of small, possibly random, perturbations of an $n\times n$ Jordan matrix, and proved that most, but not necessarily all, of the eigenvalues lie close to a certain circle as $n\to\infty$. In this paper we consider a similar question when the perturbation is a sparse matrix, whose entries need not be small.

It turns out that this problem is naturally associated with a type of directed graph that is relevant to a variety of problems involving unidirectional flows along one-dimensional channels between several junctions. Since the entity flowing may be a fluid (e.g. blood flow through the network of veins and arteries), traffic or electronic data, the ideas here have potentially wide applicability. The results in this paper also provide asymptotic spectral information for a special class of large directed graphs.

The spectral theory of a directed, acyclic graph is not particularly interesting. By labelling the vertices appropriately one may make the adjacency matrix of the graph upper triangular, so its eigenvalues coincide with its diagonal elements. This correctly suggests that the interesting spectral properties of directed graphs are highly dependent on the the structure of its cycles.

Let $S$ be a finite set and $A_{i,j}$ real or complex numbers indexed by $i,j\in S$. We make $S$ into a directed graph $(S,\to)$ by putting $i\to j$ if $A_{i,j}\not= 0$. We will see that $S$ can be partitioned into disjoint `junctions' $J_1,J_2,..., J_k$ linked by one-dimensional `channels' $C_1,C_2,...,C_h$. A natural way of constructing the graph is to draw the junctions as disjoint localized regions, and then join them by directed channels. However, one may also represent the channels by a sequence of horizontal lines, possibly of varying lengths, all directed from left to right. The junctions collectively may then be regarded as quasi-periodic boundary conditions that join the right ends back to the
left. In many of the applications the entries of $A$ are non-negative and the spectrum of $A$ may also be investigated by using the Perron-Frobenius theory, \cite{LOTS}. The theorems in this paper do not depend on this assumption, which is not appropriate in some contexts.

Given the above assumptions, we construct a new graph $(S^{(n)},\to)$ for every natural number $n$ as follows. We leave the junctions unaltered and replace each channel $C_i$ by a new channel $C_i^{(n)}$ with the same endpoints but with $\#(C_i^{(n)})=n\times\#(C_i)$. Under suitable assumptions the matrix $A$ induces an associated matrix $A^{(n)}$ whose coefficients are parametrized by pairs of points in $S^{(n)}$. Our goal is to investigate the asymptotic behaviour of $\Spec(A^{(n)})$ as $n\to\infty$.

The continuous analogue of our problem replaces each channel by a continuous bounded interval within which the relevant operator is $aD+bI$, where $D$ is first order differentiation and the constants $a,\, b$ vary from one interval to another. Instead of junctions one has to specify suitable boundary conditions connecting certain groups of end points. In contrast to \cite{carlson}, we do not restrict attention to self- or skew-adjoint problems. We finally mention that the wave equation on an undirected graph can be modelled in these terms by replacing each unoriented edge by a pair of oriented edges with the same ends and associating the operators $\pm D$ with the two edges; see \cite{carlson2}.

Every constant in this paper is independent of the asymptotic parameter $n$.
\section{Zeros of some analytic functions}

The material in this section is needed for the proofs of many of the later theorems, but does not mention graphs explicitly. The results that we attain are to some extent analogous to calculations in \cite{bottcher, hs, partington, davies, DH} but the details are different.

We will often need to count the number of zeros that an analytic function has in a particular region. The following elementary result enables us to use a change of variable to simplify this task.

\begin{lemma}\label{conformal} Let $U$ and $V$ be open subsets of $\C$ and $g: U \to V$ conformal. If $F: V \to \C$ is analytic then $F\circ g$ has a zero of order $m$ at $u_0 \in U$ if and only if $F$ has a zero of order $m$ at $g(u_0)$.
\end{lemma}

In the proofs to follow we shall make repeated use of a particular type of contour which is perhaps best described as being the boundary of a sector of an annulus centred at the origin. The contour $C_{r,R,\theta_1, \theta_2}$ is completely specified by four parameters $r, R, \theta_1, \theta_2$ where $r < R$ correspond to the two radii of the annulus and $\theta_1 < \theta_2$ correspond to the angular sweep of the sector. More precisely, the contour $C_{r,R, \theta_1, \theta_2} = \{ \gamma_1, \gamma_2, \gamma_3, \gamma_4\}$ is the concatenation of the following four curves

\begin{center}
\begin{tabular}{ll}
$\gamma_1(t) = te^{i \theta_1}$ & $r \leq t \leq R$ \\
$\gamma_2(t) = Re^{it}$ & $\theta_1 \leq t \leq \theta_2$ \\
$\gamma_3(t) = (r + R - t)e^{i \theta_2}$ & $r \leq t \leq R$ \\
$\gamma_4(t) = re^{i (\theta_1 + \theta_2 - t)}$ & $\theta_1 \leq t \leq \theta_2$ \\
\end{tabular}
\end{center}
We note that $C_{r,R, \theta_1, \theta_2}$ encloses the open region $\{ z : r < |z| < R, \theta_1 < \arg z <\theta_2 \}$.

Lemmas~\ref{limset} and \ref{limset2} provide examples of the more general analysis that starts with Theorem~\ref{mostgeneral}.

\begin{lemma}\label{limset} Let $p_n$ denote the polynomial
\begin{equation}
p_n(z) = (z-a)^n (z-b)^n + \alpha (z-a)^n + \beta (z-b)^n + \gamma \label{polyn}
\end{equation}
where $a \neq b$. If $\alpha$, $\beta$ and $\gamma$ are all non-zero then the roots of $p_n$ converge as $n \to \infty$ to the union of the two circles $ |z - a| = 1$ and $|z - b | = 1$.
\end{lemma}
\begin{proof}
Given $\epsilon > 0$ put
$$S_\epsilon = \{ z : |z-a| \geq 1 + \epsilon\} \cap \{ z : |z-b| \leq 1 - \epsilon \}.$$
If $z \in S_\epsilon$ then
$$ \frac{p_n(z)}{(z-a)^n} - \alpha = (z-b)^n + \beta \left( \frac{z-b}{z-a} \right)^n + \frac{\gamma}{(z-a)^n}.$$
Therefore $$ \left| \frac{p_n(z)}{(z-a)^n} - \alpha \right| \leq \frac{|\alpha|}{2} $$ for all large enough $n$, uniformly with respect to $z \in S_\epsilon$. This implies that $p_n(z) \neq 0$. By carrying out three other similar calculations we find that every root of $p_n$ lies in the $\epsilon$-neighbourhood $N_\epsilon$ of the union of the two circles provided $n$ is large enough.

We now show that if $z_0$ lies in the union of the two circles then there exists a sequence $z_n$ of complex numbers such that $z_n \to z_0$ and  $p_n(z_n) = 0$ for all $n$. There are essentially two cases to consider, depending on whether $z_0$ lies on one or both of the circles $|z-a|=1$ and $|z-b|=1$.

Firstly, suppose $z_0$ lies on only one of the circles. For the sake of definiteness we suppose that $|z_0 - a| = 1$ and $|z_0 - b| <1$. Let $U$ be a small open neighbourhood of $z_0$ so that $ | (z-a) (z-b) | < c < 1$ and $ |z-b| < c < 1$ for all $z$ in $U$. This implies that $q_n(z) = (z-a)^n (z-b)^n + \beta (z-b)^n$ is uniformly exponentially small for $z$ in $U$. This suggests that inside $U$ the roots of
$$ p_n(z) = \alpha (z - a)^n + \gamma + q_n(z)$$
should be close to the roots of $f_n(z) = \alpha (z-a)^n + \gamma$.

Suppose $z_0 = e^{i \phi_0} + a$. Let $w_n = | \gamma / \alpha |^{1/n} e^{i \phi_n} + a $ be a root of $f_n(z)$ which is closest to $z_0$. Since $| \gamma / \alpha |^{1/n} \to 1$ and $\phi_n \to \phi_0$ we see that $w_n \to z_0$. Let $C_n$ be the contour $C_{ r_n, R_n, \theta_{1,n}, \theta_{2, n}} + a$ where $r_n = ( \frac{1}{2} | \gamma / \alpha|)^{1/n}$, $R_n = ( \frac{3}{2} | \gamma / \alpha |) ^ {1/n}$, $\theta_{1,n} = \phi_n - \pi/2n$ and $\theta_{2,n} = \phi_n + \pi / 2n$. The region enclosed by $C_n$ contains $w_n$.

Let $M = |\gamma| / 2$. We show that $z \in C_n$ implies $|f_n(z)| \geq M$, where we note the bound is independent of $n$. Since $C_n = \{ \gamma_1, \gamma_2, \gamma_3, \gamma_4\}$ there are four cases to consider. If $z \in \gamma_1$ then $ z = \rho e^{i \theta_{1,n}} + a$ where $ r_n \leq \rho \leq R_n$ and so
\begin{eqnarray*}
|f_n(z)| & = &  | \alpha \rho^n e^{i(n\phi_n - \pi/2)} + \gamma| \\
        & = & | \alpha \rho^n e^{i(n\phi_n - \pi/2)} + \gamma - (\alpha |\gamma / \alpha| e^{i n \phi_n} + \gamma)| \\
        & = & |\alpha| |i(-\rho^n) - |\gamma / \alpha|| \\
        & \geq & |\gamma|
\end{eqnarray*}
where the second equality uses $\alpha |\gamma / \alpha| e^{i n \phi_n} + \gamma = f(w_n) = 0$. The case $z \in \gamma_3$ is similar. If $z \in \gamma_2$ then $z = R_n e^{i \theta} + a$ for some $\theta$ and so
\begin{eqnarray*}
|f_n(z)| & = & \left | \alpha \frac{3}{2} | \gamma / \alpha| e^{in\theta} + \gamma \right | \\
        & \geq & \frac{3}{2}|\gamma| - |\gamma| \\
        & = & \frac{1}{2} |\gamma|
\end{eqnarray*}
The case $z \in \gamma_4$ is similar. Therefore $|f_n(z)| \geq M$ for all $z \in C_n$.

By taking $N$ large enough we can ensure that for all $n > N$ that $C_n$ is completely contained in $U$ and that $|q_n(z)| < M/2$ for all $z$ in $U$. Therefore, for $n > N$ and $z \in C_n$ we have
$$ |q_n(z)| < M/2 < |f_n(z)| $$
Hence by Rouche's theorem $f_n(z)$ and $f_n(z) + q_n(z)$ have the same number of roots inside $C_n$. Therefore $p_n(z)$ has a root $z_n$ inside $C_n$. Furthermore, $z_n \to z_0$ because
$$ r_n < |z_n| < R_n $$
$$ \theta_{1,n} < \arg z_n < \theta_{2,n} $$
and $r_n, R_n \to 1$ and $\theta_{1,n}, \theta_{2,n} \to \phi_0$.

The remaining cases for when $z_0$ lies on one of the circles are very similar, but may involve an extra step. For example, if $|z_0 - a| = 1$ and $|z_0 - b| > 1$ we apply the above analysis instead to the equation $p_n(z)/(z-b)^n = 0$ with
 $$f_n(z) = (z - a)^n + \beta \ \ {\rm and} \ \ q_n(z) = \alpha \left(\frac{z-a}{z-b} \right)^n + \gamma \left(\frac{1}{z-b} \right)^n $$

The final case is when $z_0$ lies on both circles, that is $|z_0 - a| =1$ and $|z_0 - b|=1$. We first choose a sequence $\{z_0^{(m)}\}_m$ such that $z_0^{(m)} \to z_0$ with each $z_0^{(m)}$ belonging to only one of the circles. By the previous case, for each $m$ there exists a sequence a sequence $\{z_n^{(m)}\}_n$ such that $p_n(z_n^{(m)})=0$ and $z_n^{(m)} \to z_0^{(m)}$. There exist positive integers $M_1 < M_2 < M_3 < \dots $ so that for each $m$ we have $$ |z_n^{(m)} - z_0^{(m)} | < \frac{1}{2^m} \ \ {\rm for} \ \ n \geq M_m.$$ The sequence $\{z_n\}$ defined by putting $z_n = z_n^{(m)}$ when $M_m \leq n < M_{m+1}$ satisfies $p_n(z_n) = 0$ and $z_n \to z_0$, as required.
\end{proof}

If one of the coefficients  $\alpha$ or $\beta$ vanishes then the form of the spectrum changes.  This is illustrated in Figures~\ref{fig:poly2a} and \ref{fig:poly2b} and proved in Lemma \ref{limset2}.

\begin{lemma}\label{limset2} Assume that $\alpha \neq 0$, $\beta = 0$, $\gamma \neq 0$ and $b = -a$, and that
$$p_n(z) = (z-a)^n(z+a)^n + \alpha (z-a)^n + \gamma.$$
 If $0 < a < 1$ then the roots converge as $n \to \infty$ to the union of two arcs within the circles $|z-a|=1$ and $|z+a|=1$ together with an arc in the fourth order curve $|(z-a)(z+a)|=1$. The ends of all three arcs are at $\pm i \sqrt{1 - a^2}$. However, if $a>1$ then the roots of $p_n$ converge to the entire circle $|z+a|=1$ together with the closed fourth order curve $|(z-a)(z+a)| =1$ contained in $\{z  : \Re(z) > 0\}$.
\end{lemma}

\begin{figure}[h!]
\centering
\scalebox{0.8}{\includegraphics{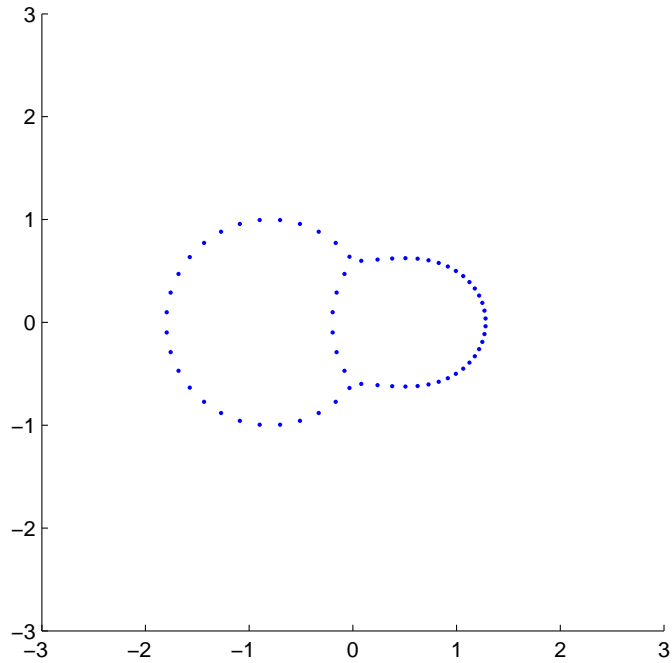}}
\caption{The constant $a$ in $p_n(z)$ of Lemma~\ref{limset2} satisfies $0<a<1$. }
\label{fig:poly2a}
\end{figure}

\begin{proof} First suppose that $0 < a < 1$. Let $a_1 = 1$, $a_2 = \alpha$, $a_3 = \gamma$ and
\begin{eqnarray*}
f_1(z) & = & (z-a)(z+a) \\
f_2(z) & = & (z-a) \\
f_3(z) & = & 1
\end{eqnarray*}
so that $$p_n(z) = a_1 f_1(z)^n + a_2 f_2(z)^n + a_3 f_3(z)^n.$$
We will prove that as $n \to \infty$ the roots of $p_n$ converge to the union of the following subsets of $\C$, which we will denote by $L$.
\begin{eqnarray*}
L_{1,2} &=& \{ z \in \C: |f_1(z)| = |f_2(z)| > |f_3(z)| \} = \{ z : |z+a| =1, \ \Re(z) < 0 \} \\
L_{1,3} &=& \{ z \in \C: |f_1(z)| = |f_3(z)| > |f_2(z)| \} =  \{ z : |(z-a)(z+a)|=1, \ \Re(z) > 0 \} \\
L_{2,3} &=& \{ z \in \C: |f_2(z)| = |f_3(z)| > |f_1(z)| \} =  \{ z: |z-a| = 1, \ \Re(z) < 0 \} \\
L_{\rm end} &=&  \{ \pm i \sqrt{1 - a^2} \}
\end{eqnarray*}
where $L_{\rm end}$ is the set of endpoints of the arcs $L_{1,2}$, $L_{1,3}$ and $L_{2,3}$. Let $L_\epsilon$ be the $\epsilon$-neighbourhood of $L$. We will show that $\C \backslash L_\epsilon$ is contained in the union of finitely many sets of the form
$$B_{r,\delta} = \{ z \in \C : |f_r(z)(1-\delta)| > \max\{|f_s(z)| : s \neq r \}.$$
Since $L_\epsilon$ is bounded and since $f_1(z)$ dominates both $f_2(z)$ and $f_3(z)$ in absolute value as $|z| \to \infty$, there exists $R>0$ and  $1 > \delta_1 >0$ such that $L_\epsilon \subset B(0,R)$ and $z \notin B(0,R)$ implies
$$ |f_1(z)(1-\delta_1)| > \max \{ |f_2(z)|, |f_3(z)| \}. $$
Therefore $\C \backslash B(0,R) \subseteq B_{1,\delta_1}$. Let $K = \bar{B}(0,R) \backslash L_\epsilon$. For each $z \in K$ there is an $r_z \in \{1,2,3\}$ and $1 > \delta_z > 0$ such that
$ z \in B_{r_z, \delta_z}$. Since each $B_{r_z, \delta_z}$ is open and $K$ is compact, we can cover $K$ with finitely many such $B_{r_z, \delta_z}$. In fact, if we take $\delta$ small enough we conclude that
$$ \C \backslash L_\epsilon \subseteq \bigcup_{r=1}^3 B_{r,\delta}.$$

Inside the set $B_{r, \delta}$ the term $a_r f_r(z)^n$ dominates the remaining terms of $p_n(z)$ so that for large enough $n$ we have $p_n(z) \neq 0$ for all $z \in B_{r,\delta}$. More precisely, if $z \in B_{r,\delta}$ then $f_r(z) \neq 0$ and

\begin{eqnarray*}
\left | \frac{p_n(z)}{a_r f_r(z)^n} \right | & = &\left |1 + \sum_{s \neq r} (a_s/a_r) \left (\frac{f_s(z)}{f_r(z)} \right)^n \right | \\
        & \geq & 1 - \sum_{s \neq r} |a_s/a_r| (1 - \delta)^n \\
        & > & 0
\end{eqnarray*}
for $n$ large enough, where we note that the bound is uniform for $z \in B_{r,\delta}$. Consequently $p_n(z) \neq 0$ for all $z \in B_{r, \delta}$ for $n$ large enough. Since $\C \backslash L_\epsilon$ is contained in the union of finitely many such $B_{r, \delta}$ all the zeros of $p_n(z)$ are in $L_\epsilon$ for $n$ large enough, as required.

We now show that if $z_0$ is in $L$ then there is a sequence $\{z_n\}$ such that $p_n(z_n) = 0$ and $z_n \to z_0$. There are number of cases to consider, but the proofs are similar. We prove the case $z_0 \in L_{1,3}$ in a manner which emphasises the general technique. Since $z_0 \in L_{1,3}$ we have
$$ |f_1(z_0)| = |f_3(z_0)| > |f_2(z_0)|$$
By continuity there is an open neighbourhood $U$ of $z_0$ such that $|f_2(z)/f_3(z)| < c < 1$ for all $z \in U$.  We want to solve $p_n(z) = 0$ for $z$ close to $z_0$, that is
$$ a_1 f_1(z)^n + a_2 f_2(z)^n + a_3 f_3(z)^n = 0. $$
After dividing both sides by $a_1$ and $f_3(z)^n$ this is equivalent to solving
$$ {\left ( \frac{f_1(z)}{f_3(z)} \right)}^n + \frac{a_3}{a_1} + \frac{a_2}{a_1}{ \left ( \frac{f_2(z)}{f_3(z)} \right)^n} =0$$
Putting $f(z) = f_1(z)/f_3(z)$, $a = a_3/a_1$ and $g_n(z) = (a_2/a_1)(f_2(z)/f_3(z))^n$ we want to solve
\begin{equation}
f(z)^n + a + g_n(z) = 0 \label{eq1}
\end{equation}
for $z$ near $z_0$. We note that $f'(z_0) \neq 0$,  $f(z_0) = e^{i \theta_0}$ for some $\theta_0$ and $g_n(z)$ is uniformly exponentially small in $U$ because $|g_n(z)| \leq |a_2/a_1| c^n$ where $c<1$.

Since $f'(z_0) \neq 0$ the inverse mapping theorem implies that there exist open sets $V$ and $W$ such that $z_0 \in V$, $f(z_0) \in W$ and $f : V \to W$ is conformal with analytic inverse $g : W \to V$. By reducing $V$ if necessary we may assume that $V \subseteq U$.

Under the change of variable $w = f(z)$ equation (\ref{eq1}) becomes
$$ w^n + a + \tilde{g_n}(w) = 0 $$
where $\tilde{g_n}(w) = g_n(g(w))$. We want to solve this equation for $w$ near $f(z_0)$. Since $\tilde{g_n}(w)$ is uniformly exponentially small in $W$ as $n \to \infty$ we expect the solutions to be close to those of
\begin{equation}
w^n + a = 0 \label{eq2}
\end{equation}
by Rouche's theorem.  Let $w_n = |a|^{1/n} e^{i \phi_n} $ be a solution of (\ref{eq2}) with argument $\phi_n$ closest to $\theta_0$. This ensures $w_n \to f(z_0)$ as $n \to \infty$. Let $C_n$ be the contour $C_{r_n, R_n, \theta_{1,n}, \theta_{2,n}}$ with $r_n = (\frac{1}{2}|a|)^{1/n}$, $R_n = (\frac{3}{2}|a|)^{1/n}$, $\theta_{1,n} = \phi_n - \frac{\pi}{2n}$ and $\theta_{2,n} = \phi_n + \frac{\pi}{2n}$. We note that $w_n$ lies in the interior of $C_n$. Routine estimates show that $w \in C_n$ implies
$$ |w^n + a| \geq |a|/2 $$
for all $n$. For $n$ large enough we have $C_n$ and its interior completely contained in $W$. Since $\tilde{g_n}(w) \to 0$ uniformly as $n \to \infty$ for $w \in W$, for $n$ large enough we have
$$ |\tilde{g_n}(w)| < |a|/2 \leq |w^n + a|$$
for all $w \in C_n$. Therefore by Rouche's theorem  $w^n + a + \tilde{g_n}(w)=0$ has a solution $u_n$  inside $C_n$. Let $z_n = g(u_n)$. The sequence $\{z_n\}$ satisfies $p_n(z_n)=0$ and $z_n \to z_0$, as required.

\begin{figure}[h!]
\centering
\scalebox{0.8}{\includegraphics{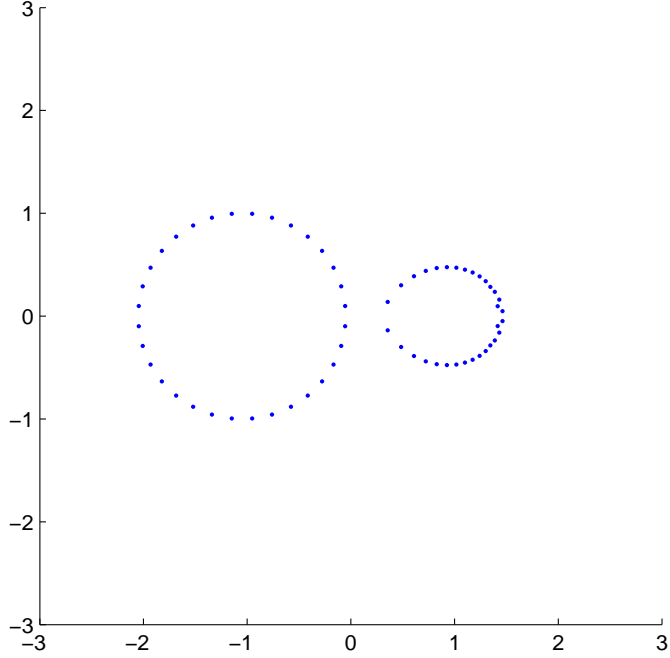}}
\caption{The constant $a$ in $p_n(z)$ of Lemma~\ref{limset2} satisfies $a>1$.}\label{fig:poly2b}
\end{figure}

The case $a>1$ is handled very similarly, with the only difference being the set $L$ which is now the union of the arcs
\begin{eqnarray*}
L_{1,2} &=& \{ z \in \C: |f_1(z)| = |f_2(z)| > |f_3(z)| \} = \{ z : |z+a| =1\} \\
L_{1,3} &=& \{ z \in \C: |f_1(z)| = |f_3(z)| > |f_2(z)| \} =  \{ z : |(z-a)(z+a)|=1, \ \Re(z) > 0 \} \\
L_{2,3} &=& \{ z \in \C: |f_2(z)| = |f_3(z)| > |f_1(z)| \} =  \emptyset \\
L_{\rm end} &=& \emptyset
\end{eqnarray*}
where we note that $L_{1,3}$ is a closed fourth order curve which is actually completely contained in the open disk $B(a, 1)$. The proof now continues as before.
\end{proof}

Lemma~\ref{asympcircle} is not covered by the general analysis and may be viewed as a degenerate case.

\begin{lemma}\label{asympcircle} Let $p$ be a polynomial in two variables of the form
$$ p(z,w) = a_m(z)w^m + a_{m-1}(z)w^{m-1} + \dots + a_1(z)w + a_0(z).$$
Let $\epsilon > 0$. As $n\to \infty$ the solutions of $p(z, z^n) = 0$ that satisfy $ |z| < 1 - \epsilon$ converge to the solutions of $a_0(z) = 0$ satisfying $ |z| < 1 - \epsilon$, while the solutions of $p(z, z^n) = 0$ that satisfy $|z| > 1 + \epsilon$ converge to the solutions of $a_m(z) = 0$ satisfying $|z| > 1 + \epsilon$.

Assuming that neither $a_0(z) = 0$ nor $a_m(z) = 0$ has a solution of modulus $1$, the remaining solutions of $p(z,z^n) = 0$ are asymptotically uniformly distributed around the circle $ |z| = 1$ in the sense that for all $\alpha$ and $\beta$ such that $ \beta - \alpha \leq 2 \pi $ and for all $ \delta > 0$ we have
$$ \frac{\# \{ z : p(z, z^n) = 0, 1 - \delta < |z| < 1 + \delta, \alpha < \arg z < \beta \}}
{\# \{ z : p(z, z^n) = 0 \}}
 \to \frac {\beta - \alpha} {2\pi} \ \  {\rm as} \ \ n \to \infty$$
where we count the roots with multiplicity.
\end{lemma}

\begin{proof}

We note that $p(z, z^n)$ has $nm + d$ roots counted with multiplicity, where $d$ is the degree of $a_m(z)$. We begin by proving the last statement. Fix an arbitrary point $z_0$ of modulus $1$. We first prove that the result is true in a neighborhood of $z_0$. The equation $p(z_0, w) = 0$ has exactly $m$ solutions counted with multiplicity because $a_m(z_0) \neq 0$, and each solution is non zero because $p(z_0, 0) = a_0(z_0) \neq 0$. Let $w_1, \dots, w_k$ be the distinct solutions and let $m_1, \dots, m_k$ be the corresponding multiplicities so that $m_1 + \dots + m_k = m$. The $nm$ solutions of $p(z_0, z^n)=0$ are precisely the $n^{\rm th}$ roots of the $w_i$, that is
$$ \{ z : p(z_0, z^n) = 0 \} = \bigcup_{i=1}^{k} \{ z : z^n = w_i \}$$
Let $z_i$ be an $n^{\rm th}$ root of $w_i$. We show that the solution $z_i$ of $p(z_0, z^n) = 0$ has multiplicity $m_i$. Substituting $z^n$ for $w$ and $z_i^n$ for $w_i$ in
$$ p(z_0, w) = a_{m_i}(w - w_i)^{m_i} + a_{m+1}(w - w_i)^{m_i + 1} + \dots \hspace{10mm}  (a_{m_i} \neq 0)$$
yields
\begin{eqnarray*}
p(z_0, z^n) &=& a_{m_i}(z^n - z_i^n)^{m_i} + a_{m+1}(z^n - z_i^n)^{m_i +1} + \dots \\
          &= & (z - z_i)^{m_i}(z^{n -1} + z^{n -2}z_i + \dots + z_i^{n -1})^{m_i}(a_{m_i} + a_{m+1}(z^n - z_i^n) + \dots) \\
          &=& (z - z_i)^{m_i} g(z)
\end{eqnarray*}
where $g(z)$ is analytic and satisfies $g(z_i) \neq 0$, and so the root $z_i$ has multiplicity $m_i$.

Let $C_1, \dots, C_k$ be contours of the form $C_{r,R,\theta_1, \theta_2}$ such that the regions enclosed by the $C_i$ are disjoint and such that $w_i$ lies inside the region enclosed by $C_i$. The preimage of $C_i$ under the map $z \mapsto z^n$ consists of $n$ disjoint contours of the form $C_{r,R,\theta_1, \theta_2}$ which are equally spaced around the unit circle in the sense that they may be labelled $C_i^{1,n}, \dots, C_i^{n,n}$ so that
$$ C_i^{j,n} = \zeta^{ j-1} C_i^{1,n} \ \  {\rm for} \ \  j = 1,\dots, n \ \ {\rm where} \ \ \zeta = e^{i 2\pi / n}$$
where multiplication by $\zeta$ effectively rotates the contour by an angle of $2\pi / n$.
 As $n \to \infty$ the radii of the preimage contours converge to $1$ and their angular sweep converge to $0$. We may label the $n^{\rm th}$ roots of $w_i$ as $z_i^{1, n}, \dots, z_i^{n,n}$ so that $z_i^{j,n}$ is inside the contour $C_i^{j,n}$.

 Let $C = \cup_{i=1}^{k} C_i$. If $w \in C$ then $p(z_0, w) \neq 0$ since the only zeros of this polynomial are the $w_i$. Since $C$ is compact there is a constant $M_1 > 0$ such that $|p(z_0, w)| > M_1$ for all $w \in C$. Note that if $z \in C_i^{j, n}$ then $z^n$ is on $C$ and so $|p(z_0, z^n)| > M_1$. Let $$M_2 = \max \{ |w^k| : w \in C, \ 0 \leq k \leq m \} > 0$$

The continuity of the polynomials $a_i(z)$ imply there exists an open neighborhood of $z_0$, say $A$, such that $$ |a_i(z) - a_i(z_0)| \leq \frac {M_1} {2(m+1)M_2} \ {\rm for} \ 0 \leq i \leq m \ {\rm and} \ z \in A$$
By reducing $A$ if necessary, we may assume that $$A = \{ z : s < |z| < S, \  \phi_1 < \arg z < \phi_2 \}$$

Inside the contour $C_i^{j, n}$ the equation $p(z_0, z^n)=0$ has precisely $m_i$ solutions, and these are all at the point $z_i^{j,n}$. If the contour $C_i^{j,n}$ lies completely inside the region $A$ and if $z \in C_i^{j,n}$ then

\begin{eqnarray*}
|p(z_0, z^n) - p(z, z^n)| & \leq & |a_m(z_0) - a_m(z)||z^{nm}| + \dots + |a_1(z_0) - a_1(z)||z^n| + |a_0(z_0) - a_0(z)| \\
                & \leq & (m+1) \frac {M_1}{2(m+1)M_2} M_2 \\
                & < & M_1 \\
                & \leq & |p(z, z^n)| \\
\end{eqnarray*}

Therefore by Rouche's theorem the equation $p(z, z^n) = 0$ has precisely $m_i$ solutions inside the contour $C_i^{j,n}$. Since the angle between successive contours $C_i^{j,n}$ is $2 \pi / n$ we conclude $$ \#\{j: C_i^{j, n} \subset A\} = \frac{n(\phi_2 - \phi_1)}{2\pi} + O(1)$$
as $n \to \infty$ where the error term $O(1)$ is actually $\leq 2$ for all large enough  $n$. Consequently the number of solutions of $p(z, z^n) = 0$ that lie in $A$ is at least
$$ \sum_{i=1}^{k} m_i \times  \#\{j: C_i^{j, n} \subset A\}   = \frac{nm(\phi_2 - \phi_1)}{2\pi} + O(1)$$

Since the unit circle $\{ z : |z|=1 \}$ is compact we may cover it with finitely many sets of the form $A$, say $A_1, \dots,A_N$, such that in each $A_i$ we have the above lower bound on the number of solutions of $p(z,z^n)=0$ in $A_i$. After trimming these sets and relabeling we may assume that for $i = 1, \dots,N$ that
$$ A_i = \{ z : 1 - \delta_0 < |z| < 1 + \delta_0, \  \phi_{i-1} < \arg z < \phi_i \} $$
where $\delta_0 < \delta$ and $\phi_0 < \phi_1 < \dots < \phi_N$ with $\phi_N - \phi_0 = 2\pi$. For $n$ large enough we conclude that the number of solutions of $p(z,z^n) = 0$ in the annulus $\{ z : 1-\delta_0 < |z| < 1 + \delta_0 \}$ is at least
$$ \sum_{i=1}^N \frac {nm(\phi_i - \phi_{i-1})}{2\pi}+O(1) = nm + O(1) $$
Since $p(z, z^n) = 0$ has only $nm+d$ solutions we conclude that all but $O(1)$ of these solutions lies inside a contour $C_i^{j,n}$. Therefore, the number of solutions of $p(z, z^n) = 0$ in a region $T$ can be counted with increasing accuracy as $n \to \infty$ by identifying those contours $C_i^{j,n}$ which are contained in $T$ and recalling that the region enclosed by $C_i^{j,n}$ contains exactly $m_i$ solutions. Therefore

\begin{eqnarray*}
\frac{\# \{ z : p(z, z^n) = 0, 1 - \delta < |z| < 1 + \delta, \alpha < \arg z < \beta \}}
{\# \{ z : p(z, z^n) = 0 \}} & = &
\frac{ \sum_{i=1}^k {nm_i(\beta - \alpha)} /{2 \pi} +O(1)} {nm+d}  \\
                & = & \frac{nm(\beta - \alpha)/2\pi + O(1)}{nm+d} \\
                & \to & \frac {\beta - \alpha} {2 \pi} \\
\end{eqnarray*}
as $n \to \infty$, as required.

We now return to the first statement. Let $B = \{ z : |z| < 1 - \epsilon \}$ and $b_n(z) = p(z, z^n) - a_0(z)$ so that $p(z, z^n) = b_n(z) + a_0(z)$.  Let $z_1, \dots, z_k$ be the roots of $a_0(z)$ in $B$. Without loss of generality we may suppose that $\epsilon$ is small enough so that all the roots of $a_0(z)$ with modulus less than $1$ are in $B$.  For $i = 1, \dots, k$ let $D_i$ be a disk of radius $\delta$ and centre $z_i$, and let $S_i$ be its circular boundary. For $\delta$ small enough each $D_i \subseteq B$. There exists $M>0$ such that $ |a_0(z)| > M$ for all $z$ in $B \backslash \cup_i D_i$. Since $b_n(z) \to 0$ uniformly on $B$ as $n \to \infty$, we have for large enough $n$ that
$$|a_0(z)| > M/2 > |b_n(z)|$$ for all $z$ in  $B \backslash \cup_i D_i$. Consequently the only solutions of $p(z,z^n) = 0$ in $B$ are in $\cup_i D_i$. Furthermore, Rouche's theorem applied to the contour $S_i$ implies that each $D_i$ contains a solution of $p(z, z^n) = 0$. Since $\delta$ is arbitrary the result follows. The proof of the second statement is similar.

\end{proof}

\begin{example}\label{tworings} Consider the polynomial $p$ given by
$$ p(z, w) = w^2 + a_1(z)w + a_0(z)$$
where $a_1(z) = -z^2 - z + 9/2$ and $a_0(z) = z^3 - z^2/2 - 4z + 2$. Letting $z = e^{i\theta}$ and solving for $w$ gives two distinct solutions
\begin{eqnarray*}
f_+(e^{i\theta}) &=& e^{i\theta} - 1/2 \\
f_-(e^{i\theta}) &=& e^{2i\theta} - 4
\end{eqnarray*}
for all $\theta \in [0, 2\pi]$. Since $1/2 \leq |f_+(e^{i\theta})| \leq 3/2$ and $3 \leq |f_-(e^{i\theta})| \leq 5$ we   should anticipate that for large $n$ the roots of $p(z,z^n) = 0$ will form two distinct rings of solutions, both close to the unit circle, along with an extra zero close to 1/2. This is illustrated in Figure~\ref{fig:fig3}.

\begin{figure}[h!]
\centering
\scalebox{0.7}{\includegraphics{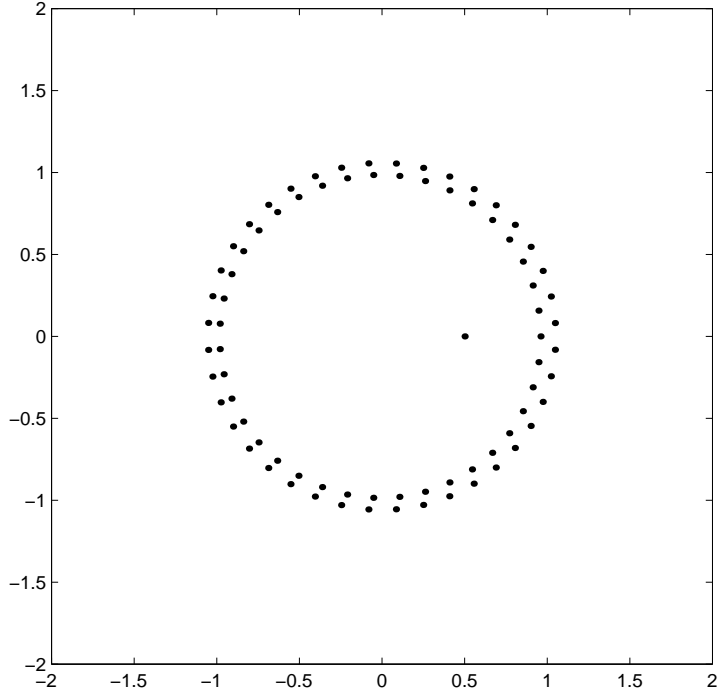}}
\caption{Solutions of $p(z,z^n) = 0$ for $n=40$ in Example~\ref{tworings}.}
\label{fig:fig3}
\end{figure}
\end{example}

\begin{example}\label{interlockrings} Consider the polynomial $p$ given by
$$ p(z,w) = w^2 -4w - 8z+3$$
Letting $z = e^{i\theta}$ and solving for $w$ gives
$$ w = 2 \pm \sqrt{8 e^{i\theta} + 1} $$
For each $\theta$ let $w_1(\theta)$ and $w_2(\theta)$ denote the two roots. The function $f: [0, 2\pi] \to \R$ defined by
$$ f(\theta) = | |w_1(\theta)| - |w_2(\theta)|| $$
measures the difference in absolute value of the two roots when $z = e^{i \theta}$, and is well defined despite their being no canonical choice for $w_1(\theta)$ and $w_2(\theta)$. For each $\theta$ there are two (possibly overlapping) rings of solutions of $p(z,z^n) = 0$ in the vicinity of $e^{i \theta}$ corresponding to those roots close to the solutions of $z^n = w_1(\theta)$ and $z^n = w_2(\theta)$. If $f(\theta_1) > f(\theta_2)$ the two rings of solutions near $e^{i \theta_1}$ are further apart then when near $e^{i \theta_2}$ for fixed $n$. A straightforward calculation shows that
$$f(\theta) = 2(16 \cos \theta + 65)^{1/4}$$
and so $f$ decreases from $0$ to $\pi$ and increases from $\pi$ to $2 \pi$. Therefore one would expect the two rings of solutions to be closest together at $-1$, with the distance between the two rings gradually increasing as one moves from $-1$ to $1$ along the rings. This is illustrated in Figure~\ref{fig:fig4} for $n = 20$.

\begin{figure}[h!]
\centering
\scalebox{0.7}{\includegraphics{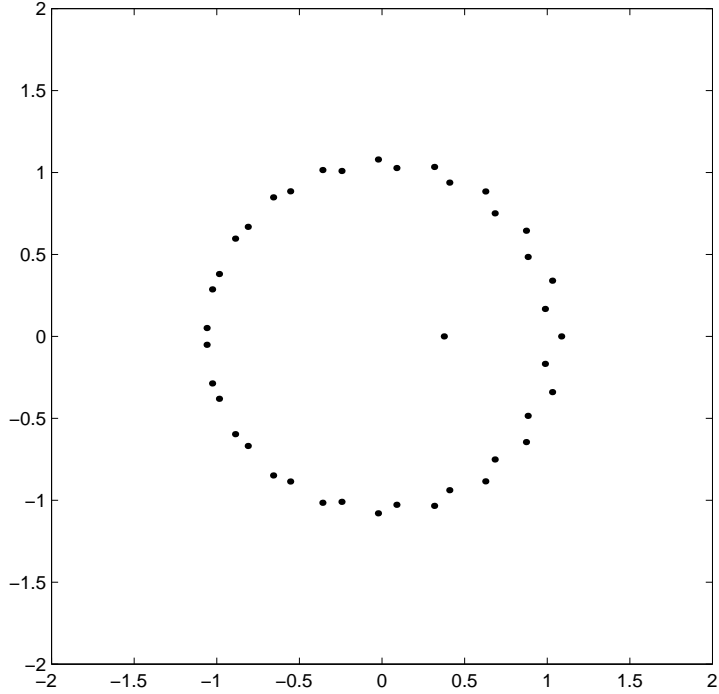}}
\caption{Solutions of $p(z,z^n) = 0$ for $n=20$ in Example~\ref{interlockrings}.}
\label{fig:fig4}
\end{figure}
\end{example}

The correct context for understanding Lemmas~\ref{limset} and \ref{limset2}  involves functions of the type
\begin{equation}\label{generalform} F_n(z) = \sum_{r=1}^m a_r(z) f_r(z)^n
\end{equation}
where $\{f_r \}_{r=1}^m$ and $\{a_r\}_{r=1}^m$ are non zero analytic functions on an open connected set $R$. We impose the following hypotheses on the $f_r$:

\begin{enumerate}
\item There exists a compact set $K \subseteq R$ and constant $c \in (0,1)$ such that $z \in R\backslash K$ implies
    $$
    |f_r(z)/ f_1(z)| < c  \hspace{3mm}\mbox{ if $r\not= 1$.}
    $$
\item For each $r \neq s$ the function $f_r/f_s$ is not constant.
\item For all choices of the three distinct integers $r, s, t$ the set
$$ \{ z \in R : |f_r(z)| = |f_s(z)| = |f_t(z)| \} $$
is finite.
\end{enumerate}

Hypothesis 3 is false for the adjacency matrices of ordered graphs. In this degenerate case the functions $f_r$ are all powers of $z$, and there is only one curve involved, the circle $|z|=1$, as seen in Lemma \ref{asympcircle}.

The set $L$ defined in our next theorem is the union of what might be called the anti-Stokes lines of the function $F_n$. We note that Hypothesis 3 is not required for the proof and so the result holds in a more a general context.
\begin{theorem}\label{mostgeneral} Let $L \subseteq R$ be the compact set consisting of all points at which two of the functions $f_r$ have the same absolute value and no others have a larger absolute value, that is
$$ L = \bigcup_{r=1}^m \{ z \in R: |f_r(z)| = \max\{ |f_s(z)| : s \neq r \} \} .$$
Then given $\epsilon > 0$ every zero of $F_n$ is contained in the $\epsilon$-neighbourhood $L_\epsilon$ of $L$ provided $n$ is large enough, with the exception of certain zeros which converge to certain zeros of one of the functions $a_r$ as $n \to \infty$.
\end{theorem}

\begin{proof} The set $L$ is bounded because $|f_1(z)| > |f_r(z)|$ for $r \neq 1$ and $z$ outside a compact set, by hypothesis. It is also closed because each function $|f_r(z)| - \max\{ |f_s(z)| : s \neq r \}$ is continuous in $z$. Hence $L$ is compact. The set $R \backslash L_\epsilon$ is contained in a finite union of sets of the form
$$B_{r,\delta} = \{ z \in R : |f_r(z)(1-\delta)| > \max\{|f_s(z)| : s \neq r \} \} $$
where $1 > \delta > 0$. In fact, by taking $\delta$ small enough we conclude that
$$ R \backslash L_\epsilon \subseteq \bigcup_{r=1}^m B_{r, \delta}.$$
In the set $B_{r, \delta}$ the term $a_r(z) f_r(z)^n$ dominates all the others for large enough $n$ so $F_n(z) \neq 0$, unless $z$ is close to a zero of $a_r$. We make this precise as follows. Let
$$ Z = \bigcup_{r=1}^m \{ z \in R : a_r(z) = 0 \} $$
and let $Z_{\epsilon'}$ be the $\epsilon'$-neighbourhood of $Z$, so that in fact $Z_{\epsilon'}$ is the union of the open disks $B(z_0, \epsilon')$ where $z_0$ is a root of an $a_r(z)$.
Let $K$ be a compact subset of $R$ and let $C$ be the compact set $K \backslash Z_{\epsilon'}$. Let
$$ M = \max \{ \sup_{z \in C} \left | \frac{a_i(z)}{a_j(z)} \right | : i \neq j \} < \infty $$
and choose $N$ large enough so that
$$ (m-1)M(1 - \delta)^N < \frac{1}{2}. $$
If $z \in C \backslash L_\epsilon$ then there exists $r \in \{1,\dots,m\}$ such that $z \in B_{r, \delta}$. This implies

\begin{eqnarray*}
\left |\frac{F_n(z)}{a_r(z)f_r(z)^n} \right | & = & \left | 1 + \sum_{r \neq s} \frac{a_s(z)}{a_r(z)} \left ( \frac{f_s(z)}{f_r(z)} \right )^n \right | \\
    & \geq & 1 -  (m-1)M(1-\delta)^n \\
    & \geq & \frac{1}{2}
\end{eqnarray*}
for $n>N$.  Therefore for $n$ large enough the only solutions of $F_n(z) = 0$ in $K \backslash L_\epsilon$ are within a distance $\epsilon'$ from a root of one of the $a_r$. Enlarging $K$ and decreasing $\epsilon'$ shows that the zeros of $F_n(z)$ not in $L_\epsilon$ converge to the zeros of one of the functions $a_r$ as $ n \to \infty$.
\end{proof}

We note that the zeros of $a_r$ that do not lie in $B_{r,\delta}$ for some $\delta > 0$ have no spectral significance as $n \to \infty$. If however $a_r(z_0) = 0$ and $z_0 \in B_{r, \delta}$ for some $\delta>0$ then there is a seqeunce $\{z_n\}$ such that $F_n(z_n)=0$ and $z_n \to z_0$.

A function $f : (u, v) \to \C$ is said to be real analytic if it is $C^\infty$ and if for all $a$ in $(u, v)$ there is a corresponding $\delta_a > 0$ such that
$$ f(x + a) = \sum_{n=0}^\infty \frac{f^{(n)}(a)}{n!} x^n$$
is an absolutely convergent series for $|x| < \delta_a$. Note that $f$ is real analytic if and only if $\Re f$ and $\Im f$ are real analytic, and that $f$ is real analytic if and only if $f$ is the restriction of a complex analytic function $\tilde{f} : U \to \C$ to $(u,v)$, where $U$ is some complex neighbourhood of $(u,v)$. Sums, products, quotients and compositions of real analytic functions are real analytic, and the square root of a strictly positive real analytic function is real analytic because $\log$ has an analytic branch defined on $\C \backslash [- \infty, 0]$.

The next theorem illuminates the structure of $L$. The following subsets of $R$ will play an important role. We put $A = A_1 \cup A_2 \cup A_3$ where
\begin{eqnarray*}
A_1 &=& \{ z \in R : (f_r/f_s)'(z) = 0 \ {\rm for \ some \ distinct} \ r, \ s \} \\
A_2 &=& \{ z \in R : a_r(z) = 0 \ {\rm for \ some \ integer} \ r \} \\
A_3 &=& \{ z \in R : |f_r(z)| = |f_s(z)| = |f_t(z)| \ {\rm for \ some \ distinct} \ r, \ s, \ t \} \\
\end{eqnarray*}

\begin{theorem}\label{structure} The set
\[
 L = \bigcup_{r=1}^m \{ z \in R : |f_r(z)| = \max\{ |f_s(z)| : s \neq r \} \}
\]
is the disjoint union of sets $D$ and $C_j$, $j=1, \dots , N$, with the following properties. $D$ is a finite subset of $\C$. Each $C_j$ is a real analytic curve, i.e. the range of a real-analytic one-one function $\gamma_j : (u_j, v_j) \to \C$ whose derivative does not vanish anywhere. These curves satisfy the following properties:
\begin{enumerate}
\item $\gamma_j$ is contained in  $$ \{ z \in R : |f_r(z)| = |f_s(z)| > |f_t(z)| \}$$ for distinct integers $r$, $s$ (dependent on $j$) and all $t \neq r, \ s$. If $z \in \gamma_j$ then $(f_r/f_s)'(z) \neq 0$.
\item $\gamma_j$ may be completed to a closed curve, or its ends lie in $D \cap A$.
\item $\gamma_j \cap (A_2 \cup A_3) = \emptyset$ for $j=1,\dots,N$.
\end{enumerate}

\end{theorem}

\begin{proof} We begin by showing that the result is true for each set $$ L_r = \{ z \in R :  |f_r(z)| = \max \{ |f_s(z)| : s \neq r \}\}. $$ The first step is to show that for each $z_0$ in $L_r$ there is a corresponding open neighbourhood $U_{z_0}$ such that $L_r \cap U_{z_0} \backslash \{z_0\}$  is contained in the range of finitely many analytic arcs each of which has range contained in $L_r$. There are a number of cases to consider.

If $z_0 \notin A_3$ then there exists a unique $s$ not equal to $r$ such that
$$ |f_r(z_0)| = |f_s(z_0)| > |f_t(z_0)| $$
for all $t \neq r, s$. By continuity there exists an open disk centered at $z_0$, say $B_{z_0}$, and constant $c>0$, such that $z \in B_{z_0}$ implies
$$ |f_r(z)| > c > |f_t(z)| $$
for all $t \neq r, s$. Consequently the set $L_r \cap B_{z_0}$ consists of precisely those points $z$ in $B_{z_0}$ such that $|f_r(z)| = |f_s(z)|$, or equivalently $|f_r(z)/f_s(z)| =1$. Define $f : B_{z_0} \to \C$ by $f(z) = f_r(z)/f_s(z)$. Since $|f(z_0)| = 1$ there exists $\theta_0 \in \R$ such that $f(z_0) = e^{i \theta_0}$. There are now two cases to consider depending on whether $f'(z_0)$ is zero or non-zero.

If $f'(z_0) \neq 0$ then by the inverse mapping theorem there exist open neighbourhoods $U$ of $z_0$ and $V$ of $f(z_0)$ so that $f: U \to V$ has an analytic inverse $g: V \to U$. We may assume that $V$ is a small disk of radius $ r < 1/2$ centered at $f(z_0)$ and that $U \subseteq B_{z_0}$. There exist $u, v \in \R$ satisfying $u < v < u+2\pi$ such that
$$ V \cap S^1 = \{ e^{it} : t \in (u,v) \} $$
where $S^1$ denotes the unit circle centered at $0$. Define $\gamma : (u,v) \to \C$ by $\gamma(t) = g(e^{it})$. We note that $\gamma$ is a real analytic arc because it is the restriction of $g(e^{iz})$ to $(u,v)$. Since
$$ L_r \cap U = \{ \gamma(t) : t \in (u,v) \} $$
it is sufficient to define $U_{z_0}$ to be $U$.

If $f'(z_0) = 0$ things are a bit trickier. Suppose
$$ f(z) = a_0 + a_m (z - z_0)^m + a_{m+1} ( z - z_0)^{m+1}+ \dots $$
where $m \geq 2$ and $a_m \neq 0$. Obviously $a_0 = f(z_0) = e^{i \theta_0}$. We show that $f(z)$ essentially behaves as $z \mapsto z^m$ for $z$ close to $z_0$. We now make this precise. Let
$$ k(z) = a_m + a_{m+1}(z - z_0) + a_{m+2}(z - z_0)^2 + \dots $$
Since $k(z_0) \neq 0$ we see that in a small enough open neighbourhood of $z_0$, say $U_{z_0}$, we can take an analytic branch of $k(z)^{1/m}$. Let $r(z)$ be such a branch. Now
\begin{eqnarray*}
f(z) & = & a_0 + (z - z_0)^m (a_m + a_{m+1} ( z - z_0) + a_{m+2} (z - z_0)^2 +\dots )\\
        & = & a_0 + (z - z_0)^m k(z)    \\
        & = & a_0 + [ (z - z_0) r(z) ]^m \\
        & = & a_0 + h(z)^m
\end{eqnarray*}
where $h(z) = (z - z_0)  r(z)$ satisfies $h(z_0) = 0$ and $h'(z_0) = r(z_0) \neq 0$. Since $h'(z_0) \neq 0$ we may shrink $U_{z_0}$ so that $U_{z_0} \subseteq B_{z_0}$ and $h : U_{z_0} \to U_0$ is conformal, where $0 \in U_0 = h(U_{z_0})$. Choose $ \epsilon < 1/2$ so that $ B(0, \epsilon) \subseteq U_0$, and shrink $U_{z_0}$ again so that $U_{z_0} = h^{-1}(B(0,\epsilon))$. This enables us to factorise $f(z)$ as the composition of the following three surjective maps
$$ U_{z_0} \xrightarrow{z \mapsto h(z)} B(0, \epsilon) \xrightarrow{z \mapsto z^m} B(0, \epsilon^m)
\xrightarrow{z \mapsto z + a_0} B(a_0, \epsilon^m) $$
Let $ s : B(0, \epsilon^m) \backslash [0, -a_0] \to B(0, \epsilon)$ be an analytic branch of $z \mapsto z^{1/m}$ where the branch cut $[0, -a_0]$ is the line connecting $0$ to $-a_0$.  The $m$ analytic branches of $z \mapsto z^{1/m}$ on $B(0, \epsilon^m) \backslash [0, -a_0]$ are then given by
$$ s_k(z) = e^{\frac{2\pi i}{m} k} s(z) \ {\rm for} \ k = 1,\dots, m.$$
Recalling that $a_0 = e^{i \theta_0}$, there exist $\theta_{-}, \theta_{+} \in \R$ such that $ \theta_{-} < \theta_0 < \theta_{+} < \theta_{-} + 2\pi$ and
$$ B(a_0, \epsilon^m) \cap S^1 = \{ e^{it} : t \in (\theta_{-}, \theta_{+}) \} $$
For $k = 1,\dots,m$ define $\alpha_k : (\theta_{-}, \theta_0) \to \C$ by
$$ \alpha_k(t) = h^{-1}(s_k(e^{it} - a_0)) $$
Similarly, for $k = 1,\dots, m$ define $\beta_k : (\theta_0, \theta_{+}) \to \C$ by
$$ \beta_k(t) = h^{-1}(s_k(e^{it} - a_0)) $$
We note that $\alpha_k$ and $\beta_k$ are real analytic curves that satisfy
$$ L_r \cap U_{z_0} \backslash \{z_0\} = \bigcup_{k=1}^m\{ \alpha_k(t) : t \in (\theta_{-}, \theta_0) \} \cup \{\beta_k(t) : t \in (\theta_0, \theta_+) \} $$
and so $L_r \cap U_{z_0}$ minus the point $z_0$ is contained in the range of finitely many analytic arcs, as required.

Now we consider the cases when $z_0 \in A_3$. Since $z_0 \in L_r$, $f_r$ and at least two other functions, say $f_{s_i}$ for $i=1,\dots, k$, attain the maximum absolute value, i.e.
$$ |f_r(z_0)| = |f_{s_1}(z_0)| = \dots = |f_{s_k}(z_0)| = \max \{ |f_s(z_0)| : s \neq r \} $$
There are two cases to consider depending on whether or not $f_r(z_0)$ is $0$.  We first suppose $f_r(z_0) \neq 0$. By continuity there exists $\epsilon, c > 0$ so that if $z \in B(z_0, \epsilon)$ then
$$ |f_r(z)| > c > |f_t(z)| $$
$$ |f_{s_i}(z)| > c > |f_t(z)| \ \ {\rm for} \ i = 1, \dots, k$$
for $t \neq r, s_1, \dots , s_k$. If necessary reduce $\epsilon$ so that $B(z_0, \epsilon) \cap A_3 = \{z_0\}$. For each $i = 1, \dots , k$ there exists an open set $U_i \subseteq B(z_0, \epsilon)$ such that the set $ \{ z \in U_i \backslash \{z_0\} : |f_r(z)| = |f_{s_i}(z)| \}$  can be parameterised by finitely many analytic arcs, say $\gamma^{(i)}_1, \dots , \gamma^{(i)}_{n_i}$. Let $\gamma : (u, v) \to \C$ be one of these curves so that $|f_r(\gamma(x))| = |f_{s_i}(\gamma(x))|$ for all $x \in (u,v)$. We show that either $\gamma \subseteq L_r$ or $\gamma \cap L_r = \emptyset$. Pick $t \neq r, s_i$ and suppose that there exists $x_1, x_2 \in (u,v)$ such that
$$ |f_t( \gamma(x_1))| < |f_r( \gamma(x_1))| $$
and $$ |f_t( \gamma(x_2))| > |f_r( \gamma(x_2))| .$$
By the intermediate value theorem there is a point $x_3$ between $x_1$ and $x_2$ such that
$$ |f_t(\gamma(x_3))| = |f_r(\gamma(x_3))| $$
and so $\gamma(x_3) \in A_3$. But by construction $\gamma \cap A_3 = \emptyset$ and so this is impossible. Hence we either have
$$ |f_t(z)| < |f_r(z)| \ \ {\rm for \ all} \ z \in \gamma$$
or we have
$$ |f_t(z)| > |f_r(z)| \ \ {\rm for \ all} \ z \in \gamma.$$
If there is a $t \neq r, s_i$ such that the latter case holds then $\gamma \cap L_r = \emptyset$, otherwise $\gamma \subseteq L_r$. Let
$$ U_{z_0} = \bigcap_{i = 1}^k U_i$$
and
$$ \cC = \bigcup_{i=1}^k \{ \gamma^{(i)}_j \subseteq L_r : j= 1,\dots, n_i \} .$$
It is clear that
$$ L_r \cap U_{z_0} \backslash \{ z_0 \} \subseteq \bigcup_{\gamma \in \cC} \gamma \subseteq L_r$$
as required.

The remaining case is when $f_r(z_0) = 0$. This implies all of the functions $f_s$ vanish at $z_0$ and so one of these functions has a zero of lowest order, say $k$. For $s = 1, \dots, m$ put $ g_s(z) = f_s(z)/(z-z_0)^k$. Not all the $g_s(z)$ vanish at $z_0$. If $|g_r(z_0)| = \max \{|g_s(z_0)| : s \neq r \}$ then $g_r(z_0) \neq 0$ and so we apply one of the previous cases to the $g_s$ since for all $z$ we have
$$ |f_r(z)| = \max \{ |f_s(z)| : s \neq r \} \ \ {\rm if \ and \ only \ if} \ \ |g_r(z)| = \max \{ |g_s(z)| : s \neq r \}. $$
If equality does not hold, then there is an open neighbourhood $U_{z_0}$ of $z_0$ such that $L_r \cap U_{z_0} = \{z_0\}$.

We have now shown that for each $z$ in $L_r$ there is an open set $U_z$ containing $z$ such that $L_r \cap U_z \backslash \{z\}$ is contained in the union of finitely many analytic arcs. If $z$ is not on one of these analytic arcs then $z \in A$. Since $L_r$ is compact it can be covered by finitely many such open sets, say $U_{z_1}, \dots , U_{z_n}$, and consequently $L_r$ is the union of finitely many analytic arcs, say $\gamma_1, \dots., \gamma_n$ and a finite subset of $A$. We note that each $\gamma_i$ has non-vanishing gradient since each is of the form $(f_r/f_s)^{-1}(e^{it})$ for some analytic branch of $(f_r/f_s)^{-1}$. However, as it stands, the $\gamma_i$ might not be disjoint from themselves and this must now be rectified.

Now, suppose that the $\gamma_i$ are not disjoint. Without loss of generality we may suppose that $\gamma_1 : (u_1, v_1) \to \C$ and $\gamma_2 : (u_2, v_2) \to \C$ are not disjoint, that is, there exists $t_1 \in (u_1, v_1)$ and $t_2 \in (u_2, v_2)$ such that $\gamma_1(t_1) = \gamma_2(t_2) = w_0$  for some $w_0 \in \C$. We know that $\gamma_1(t) = (f_r/f_{s_1})^{-1}(e^{it}) $ for some $s_1$ and that $\gamma_2(t) = (f_r/f_{s_2})^{-1}(e^{it})$ for some $s_2$. If $s_1 \neq s_2$ then $w_0 \in A_3$ which contradicts $\gamma_i \cap A_3 = \emptyset$ for all $i$. Therefore $s_1 = s_2 = s$. Now
$$ e^{i t_1} = (f_r/f_s)(w_0) = e^{i t_2} $$
and so $t_1 = t_2 + 2\pi n$ for some integer $n$. Since $e^{it}$ is $2 \pi$ periodic we can reparameterise the range of $\gamma_2$ by
$$ \tilde{\gamma_2} : (\tilde{u_2}, \tilde{ v_2 }) \to \C \ \ {\rm where} \ \tilde{\gamma_2}(t) = (f_r/f_s)^{-1}(e^{it})$$
where $\tilde{u_2} = u_2 + 2 \pi n$, $\tilde{v_2} = v_2 + 2 \pi n$ and where $(f_r/f_s)^{-1}$ is the same analytic branch used in the definition of $\gamma_2$. This ensures that $\gamma_1(t_1) = \tilde{\gamma_2}(t_1) = w_0$. Let $f = f_r/f_s$ and note that $f(w_0) = e^{it_1}$. Since $f'(w_0) \neq 0$ there exist open neighbourhoods $U$ of $w_0$ and $V$ of $e^{i t_1}$ such that $f : U \to V$ has analytic inverse $ g : V \to U$. If $\delta > 0$ is small enough then for all $t$ in $(t_1 - \delta, t_1 + \delta)$ we have
$$ \gamma_1(t) = g(e^{i t}) = \tilde{\gamma_2}(t). $$
Since $\gamma_1$ and $\tilde{\gamma_2}$ are analytic and agree on a small interval we deduce that  they must also agree on all of $(u_1,v_1) \cap (\tilde{u_2}, \tilde{v_2})$ by the principle of isolated zeros. Hence we can replace $\gamma_1$ and $\tilde{\gamma_2}$ by the single analytic arc
$$\tau : (\min\{u_1, \tilde{u_2}\}, \max\{ v_1, \tilde{v_2}\} ) \to \C \ \ {\rm where} \ \
\tau(t) = \left \{ \begin{array}{ll}
            \gamma_1(t)     & \mbox{ if $t \in (u_1, v_1)$} \\
            \tilde{\gamma_2}(t) &  \mbox{otherwise}
            \end{array}
        \right .
$$

If we continue the above process of replacing two overlapping arcs by a single analytic arc
 then after finitely many repetitions we will arrive at finitely many disjoint analytic arcs, say $\tau_1, \dots , \tau_m$, such that $\cup_i \tau_i = \cup_i \gamma_i$. It is clear that for each $\tau_i$ there is a corresponding $s_i$ such that the range of $\tau_i$ is contained in
$$ \{ z \in R : |f_r(z)| = |f_{s_i}(z)| > |f_t(z)| \}$$
for all $t \neq r, s_i$.

 There is no guarantee that each $\tau_i : (u_i, v_i) \to \C$ is one-one, but this situation can be easily rectified. Suppose $\tau_i : (u_i, v_i) \to \C$ is not one-one. There then exist $a,b \in (u_i, v_i)$ with $ a < b$ such that $\tau_i(a) = \tau_i(b)$ and $\tau_i | (a,b)$ is one-one.  Since
 $$ e^{i a} = (f_r/f_{s_i})(\tau(a)) = (f_r/f_{s_i})(\tau(b)) = e^{i b} $$
 we conclude that $b - a = 2  \pi n $ for some $n$. Let $\tilde{\tau_i} : (u_i, v_i) \to \C$ be the unique $2 \pi n$ periodic function such that $ \tilde{\tau_i} | [a,b) = \tau_i | [a, b)$. If $\delta > 0$ is small enough then for $|t| < \delta$ we have
 $$ \tau(a + t) = \tau(b + t).$$
 This implies $\tilde{\tau_i}$ is analytic and so $\tau_i(t) = \tilde{\tau_i}(t)$ for all $t \in (u_i, v_i)$. Therefore we may replace $\tau_i$ by the one-one function $\tau_i | (a, b)$ and its endpoint $\tau_i(a)$. So we can assume that all the $\tau_i$ are one-one.

We conclude that for each $r=1, \dots , m$ the set $L_r$ is the disjoint union of a finite subset $D_r$ of $\C$ and one-one analytic curves $\gamma^{(r)}_1, \dots , \gamma^{(r)}_{n_r}$ that satisfy Properties $1$ and $2$ in the statement of the theorem. It is possible that for $r \neq s$ that $\gamma^{(r)}_i$ is not disjoint from $\gamma^{(s)}_j$ for some $i$ and $j$. It follows by Property 1 that $\gamma^{(s)}_j \subseteq L_r$. Hence $\gamma^{(s)}_j$ is redundant and can be deleted. We continue this process until the remaining curves are disjoint. We conclude that $L = \cup_r L_r$ is the disjoint union of a finite subset $D$ of $\C$ and one-one analytic curves, say $\gamma_1, \dots, \gamma_n$, that satisfy Properties 1 and 2.

Finally, we want each $\gamma_i$ to also satisfy Property 3, that is $\gamma_i \cap (A_2 \cup A_3) = 0$. We first note $L \cap (A_2 \cup A_3)$ is finite because $L$ is compact and $A_2 \cup A_3$ discrete. If $\gamma_i : (u_i, v_i) \to \C$ meets $A_2\cup A_3$, that is $\gamma_i \cap (A_2\cup A_3) \neq \emptyset$, then there exist finitely many points, say $a_1, \dots, a_m$, such that
$$ u_i < a_1 < a_2 < \dots < a_m < v_i$$
and $\gamma_i(t) \in (A_2 \cup A_3)$ if and only if $t \in \{ a_1, \dots , a_m \}$.  We now just replace $\gamma_i$ by $$ \gamma_i | (u_i, a_1), \ \gamma_i| (a_1, a_2), \dots , \ \gamma_i | (a_m, v_i) $$
and enlarge $D$ with the points $\gamma_i(a_k)$ for all $k$. After doing this for each $\gamma_i$, we obtain a finite subset $D$ of $\C$ and a finite collection of analytic curves that satisfy all the required properties.
 \end{proof}

We now show that each arc $\gamma: (u,v) \to \C$ in Lemma~\ref{structure} can be parameterised by arc length with the arc length parameterization also being real analytic. We begin by showing that the length of $\gamma$ is finite, that is
\[ l = \int_u^v |\gamma'(t)| \ dt < \infty.\]
This is achieved by showing that $|\gamma'(t)| = O((t-u)^a)$ as $t \to u$ for some $a > -1$, and likewise at the opposite endpoint $v$. The problematic case is when for $t$ near $u$ we have
$$\gamma(t) = h^{-1}(s(e^{it} - e^{iu})) $$
where $h$ is conformal and $s$ is a branch of $z \mapsto z^{1/m}$. Differentiating it suffices to show that
$$ |s'(e^{it} - e^{iu})| = O((t-u)^a) $$
for some $a > -1$. Now
\begin{eqnarray*}
|s'(e^{it} - e^{iu})| & = & \frac{1}{m} |e^{i(t-u)} -1|^{-(m-1)/m} \\
               & \leq & c_0 (t - u)^{-(m-1)/m} \\
\end{eqnarray*}
for some constant $c_0 > 0$ as $t \to u$, as required. Hence $l < \infty$, as required. This implies the length function $ s : (u,v) \to (0, l)$ where
$$ s(t) = \int_u^t |\gamma'(x)| \ dx $$
is well defined. Since
$$ | \gamma'(x)| = (\gamma'(x) \overline{\gamma'(x)})^{1/2} $$
and $\gamma'(x)$ does not vanish we conclude that $|\gamma'(x)|$ is real analytic. Therefore $s$ is real analytic. Since $s'(x) > 0$ for all $x \in (u,v)$ we conclude that $s^{-1}$ is real analytic by the inverse mapping theorem. Therefore the arc length parameterization  $\tilde{\gamma} : (0, l) \to \C$ of $\gamma$ defined by $\tilde{\gamma}(t) = \gamma(s^{-1}(t))$ is real analytic.

Our next theorem analyses the distribution of the zeros of $F_n$ along $L$. The proof identifies the exceptional points and the densities in the statement.


\begin{theorem}\label{asympdensity} There exists a finite subset $D$ of $L$, real-analytic curves $\gamma_r : (0, l_r) \to \C$ for $r =1,\dots, N$ parameterized by arc length with non-vanishing gradients, and bounded real-analytic densities $ \rho_r : (0, l_r) \to (0,\infty)$ with the following properties:
\begin{enumerate}
\item $L$ is the disjoint union of $D$ and the ranges of all the $\gamma_r$.
\item For every choice of $r$ and every closed subinterval $[\alpha, \beta]$ of $(0, l_r)$ we have
$$ \lim_{\epsilon \to 0} \ N_\epsilon= \int_\alpha^\beta \rho_r(s) \ ds,$$
where
$$ N_\epsilon = \lim_{n \to \infty} \ \frac{1}{n} | \{ z : F_n(z) = 0 \} \cap S_\epsilon |$$
and
$$S_\epsilon = \{ z \in \C : {\rm dist}(z, \gamma_r([\alpha, \beta])) < \epsilon \}.$$
\end{enumerate}
\end{theorem}

\begin{proof}
Let $\gamma : (0,l) \to \C$ be one of the analytic curves in $L$. It has range contained in
$$ \{ z \in R : |f_r(z)| = |f_s(z)| > |f_t(z)| \} $$
for some integers $r \neq s$ and all $t \neq r, s$. By taking a sufficiently small open neighbourhood $U$ of $\gamma([\alpha, \beta])$ we can ensure that $z \in U$ implies $a_r(z) \neq 0$, $f_s(z) \neq 0$ and
$$ \left | \frac{f_k(z)}{f_s(z)} \right | < 1 - \delta $$
for some  $\delta \in (0,1)$ and all $k \neq r, s$. Since $a_r(z)f_s(z)^n \neq 0$ for all $z \in U$ the solutions (counted with multiplicity) of $F_n(z) = 0$ in $U$ are precisely the same as those of $F_n(z)/(a_r(z)f_s(z)^n) = 0$. Now
$$ \frac{F_n(z)}{a_r(z) f_s(z)^n} = \left ( \frac{f_r(z)}{f_s(z)} \right )^n + \frac{a_s(z)}{a_r(z)} +
\sum_{k \neq r,s} \frac{a_k(z)}{a_r(z)} \left ( \frac{f_k(z)}{f_s(z)} \right )^n $$
and so with $f(z) = f_r(z)/f_s(z)$, $a(z) = a_s(z)/a_r(z)$ and $g_n(z) = \sum_{k \neq r,s} ( a_k(z)/a_r(z) )(f_k(z)/f_s(z))^n$
we are interested in solving the equation
$$  f(z)^n + a(z) + g_n(z) = 0 $$
inside $U$. We note that $g_n(z)$ is uniformly exponentially small in $U$ as $n \to \infty$.

Define the argument function $\theta : (0, l) \to \R$ by $\theta(t) = - i \log(f(\gamma(t))$  using an analytic branch of $\log f \circ \gamma$. We note that $\theta$ is real analytic and can be extended continuously to the closed interval $[0,l]$. Also, define the density function $\rho : (0,l) \to \R$ by $\rho(t) = \theta'(t)/2\pi$. Suppose $\gamma$ is the arc length parameterization of $\tau: (u,v) \to \C$ where $\tau$ is of the form $f^{-1}(e^{it})$ as in Theorem \ref{structure}. Letting $s: (u,v) \to (0,l)$ be the length function
$$s(t) = \int_u^t |\tau'(x)| \  dx$$
we recall that $\gamma(t) = \tau(s^{-1}(t))$. Consequenty $\theta(t) = s^{-1}(t) + c$ for some constant $c$ and so $\rho(t)$ is bounded if $(s^{-1})'(t)$ is bounded. It suffices to check that there exists a constant $C > 0$ such that $|s'(t)| \geq C$ for all $t \in (u,v)$. Since
$$ |s'(t)| = |\tau'(t)| = |(f^{-1})'(e^{it})| $$
where $f^{-1}$ is an analytic branch of a local inverse of $f$, it suffices to check that $f'(z)$ is bounded for $z \in \tau$ (or equivalently $\gamma$.) But if one observes the construction of $\tau$ this is clearly the case. Hence $\rho$ is bounded. Also $\rho$ is positive because $(s^{-1})'(t)$ is non zero and $s$ is increasing.

The idea of the proof is to now cover $\gamma( [\alpha, \beta] )$ with finitely many sets  with disjoint interiors such that in each set we can count the number of solutions of $F_n(z) = 0$. The boundaries of these sets contribute $O(1)$ zeros as $n \to \infty$ and so their contribution can be ignored. This will be achieved over a number of steps.

Let $z_0 \in \gamma( [\alpha, \beta] )$ and let $w_0 = f(z_0)$. Since $f'(z_0) \neq 0$ there are open neighbourhoods $V$ of $z_0$ and $W$ of $w_0$ such that $f : V \to W$ is conformal with analytic inverse $g : W \to V$. We may assume that $V$ is sufficiently small so that $V \subseteq U$ and $z \in V$ implies
$$ | a(z) - a(z_0) | < \frac{ |a(z_0)|}{8}. $$
Since $|w_0| =1$ there exists $\phi_0 \in \R$ such that $w_0 = e^{i \phi_0}$. Let $C_{w_0}$ be the contour $C_{r, R, \theta_1, \theta_2}$ where $r = 1-\delta$, $R = 1+ \delta$, $\theta_1 = \phi_0 - \delta$, $\theta_2 = \phi_0 + \delta$ where $\delta > 0$ is small enough so that $C_{w_0}$ and its interior is contained in $W$. Let $A_{w_0}$ be the union of $C_{w_0}$ and its interior, so that $A_{w_0}$ is a sector of an annulus with angular sweep $\theta_2 - \theta_1$. Finally let $N_{z_0} = g(A_{w_0})$. We show that
$$ \#\{ z \in N_{z_0}^\circ : F_n(z) = 0 \} =  \frac{ n (\theta_2 - \theta_1)}{2 \pi} + O(1) $$
as $n \to \infty$. Let $\tilde{F_n}(w) = F_n(g(w))$ for $w \in W$. Since $g: W \to V$ is conformal Lemma~\ref{conformal} implies
$$  \#\{ z \in N_{z_0}^\circ : F_n(z) = 0 \} =  \#\{ w \in A_{w_0}^\circ : \tilde{F_n}(w) = 0 \}.$$
The solutions of $\tilde{F_n}(w) = 0$ counted with multiplicity are precisely the same as those of
$$ w^n + \tilde{a}(w) + \tilde{g_n}(w) = 0$$
where $\tilde{a}(w) = a(g(w))$ and $\tilde{g_n}(w) = g_n(g(w))$. We use Rouche's theorem to show that each such solution in $A_{w_0}^\circ$ is close to a root of
$$ w^n + \tilde{a}(w_0) = 0$$
which has roots $$ w_j = |\tilde{a}(w_0)|^{1/n} e^{ i s_j} \ \ {\rm for} \ j = 1, \dots, n$$
where $s_{j+1} - s_j = 2 \pi / n$. For $j = 1, \dots, n$ define $C^{j,n}$ to be the contour $C_{r_n, R_n, \theta_{1, j}, \theta_{2, j}}$ where $r_n = (\frac{1}{2}|\tilde{a}(w_0)|)^{1/n}$, $R_n = (\frac{3}{2}|\tilde{a}(w_0)|)^{1/n}$, $\theta_{1,j} = s_j - \pi/{2n}$ and $\theta_{2,j} = s_j + \pi/{2n}$. We note that the $C^{j, n}$ approach the unit circle uniformly as $n \to \infty$ and are evenly distributed around the unit circle by an angle of $2 \pi/n$.

We show that if $C^{j, n} \subseteq A_{w_0}$ then $\tilde{F_n}(w) = 0$ has precisely one solution inside $C^{j,n}$.  Routine estimates show that if $w \in C^{j, n}$ then $|w^n + \tilde{a}(w_0)| > |\tilde{a}(w_0)|/2$, where we note that the lower bound is independent of $n$. Also, since $\tilde{g_n}(w) \to 0$ uniformly as $n \to \infty$ there exists $N>0$ such that for $n > N$ we have $|\tilde{g_n}(w)| < |\tilde{a}(w_0)| / 8$ for all $w \in W$. Therefore for $w \in C^{j,n}$ and $n > N$ we have
\begin{eqnarray*}
| w^n + \tilde{a}(w) + \tilde{g_n}(w) | & \geq & |w^n + \tilde{a}(w_0)| - |\tilde{a}(w_0) - \tilde{a}(w)| - |\tilde{g_n}(w)| \\
                    & \geq &  \frac{|\tilde{a}(w_0)|}{2} - \frac{|\tilde{a}(w_0)|}{8} - \frac{|\tilde{a}(w_0)|}{8} \\
                    & = & \frac{|\tilde{a}(w_0)|}{4} \\
\end{eqnarray*}
and
\begin{eqnarray*}
| (w^n + \tilde{a}(w) + \tilde{g_n}(w)) - (w^n + \tilde{a}(w_0))| & \leq & |\tilde{a}(w) - \tilde{a}(w_0)| + |\tilde{g_n}(w)| \\
                                & < & \frac{|\tilde{a}(w_0)|}{8} + \frac{|\tilde{a}(w_0)|}{8} \\
                                & = & \frac{|\tilde{a}(w_0)|}{4}
\end{eqnarray*}
Therefore by Rouche's theorem $w^n + \tilde{a}(w) + \tilde{g_n}(w) = 0$ and $w^n + \tilde{a}(w_0)=0$ have the same number of solutions inside $C^{j,n}$, and so $\tilde{F_n}(w) = 0$ has precisely one solution counted with multiplicity inside $C^{j,n}$.

Let $D^{j, n}$ be the open region enclosed by $C^{j,n}$. Routine estimates show that if $w \in A_{w_0} \backslash \cup_{j} D^{j,n}$ then $| w^n + \tilde{a}(w_0) | \geq |\tilde{a}(w_0)|/2$. Consequently for $w \in A_{w_0} \backslash \cup_{j} D^{j,n}$ and $n > N$ we have
$$ | w^n + \tilde{a}(w) + \tilde{g_n}(w) | > \frac{|\tilde{a}(w_0)|}{4} $$
and so $\tilde{F_n}(w) = 0$ has no solutions outside the $C^{j,n}$. Therefore
\begin{eqnarray*}
\# \{ z \in N_{z_0}^\circ : F_n(z) = 0 \} &=&   \#\{ w \in A_{w_0}^\circ : \tilde{F_n}(w) = 0 \} \\
                            & = &\# \{ j : C^{j, n} \subseteq A_{z_0} \} + O(1)  \\
                             & = & \frac { n(\theta_2 - \theta_1)}{2 \pi} + O(1)
\end{eqnarray*}
as $n \to \infty$. We also note that the boundary of $N_{z_0}$ has $O(1)$ (in fact, at most 2) solutions as $n \to \infty$. Our choice of $A_{w_0}$ (and hence $N_{z_0}$) was quite arbitrary and it will be useful to impose further conditions. There exists $t_0 \in [\alpha, \beta]$ such that $z_0 = \gamma(t_0)$. We may choose $A_{w_0}$ so that there exist $a, b$ satisfying $0 < a < t_0 < b < l$ so that
$$ S^1 \cap A_{w_0} = \{e^{i \theta(t)} : t \in [a,b] \} $$
and
$$ \gamma \cap N_{z_0} = \gamma([a,b]).$$
This uses the fact that $f(\gamma(t)) = e^{i \theta(t)}$ by definition. These sets satisfy the relation
$$ f( \gamma \cap N_{z_0} ) = S^1\cap A_{z_0}.$$
Furthermore we can impose the condition that $\theta(b) - \theta(a) < \pi/2$.

For convenience, let $A_\delta(\theta_1, \theta_2)$ denote the closed region enclosed by the contour $C_{1-\delta, 1+ \delta, \theta_1, \theta_2}$. We show that there is a $\delta > 0$ small enough and  a partition of $[\alpha, \beta]$ of the form
$$ t_0 < t_1 = \alpha < t_2 < t_3 < \dots < t_m = \beta < t_{m+1} $$
so that $f^{-1}$ has an analytic branch $g_i$ in a neighbourhood of $A_i = A_\delta(\theta(t_i), \theta(t_{i+1}))$ so that the sets $N_i = g_i(A_i)$ for $i = 0,\dots, m$ have the following properties:
\begin{enumerate}
\item The interior of $\cup_i N_i$ is an open set containing $\gamma([\alpha, \beta ])$. The interiors of each $N_i$ are disjoint.
\item For $i = 0,\dots, m$ we have $S^1 \cap A_i = \{e^{i \theta(t)} : t \in [t_i,t_{i+1}] \}$ and  $\gamma \cap N_i = \gamma([t_i,t_{i+1}])$
\item For $i = 0,\dots, m$
$$\# \{ z \in N_i^\circ : F_n(z) = 0 \} = \frac{n(\theta(t_{i+1}) - \theta(t_i))}{2 \pi} + O(1)$$
as $n \to \infty$. The boundary of $N_i$ has $O(1)$ solutions as $n \to \infty$.
\end{enumerate}

Let $z \in \gamma([\alpha, \beta])$. By the previous step there exist subsets of $\C$, denoted by $N_z$ and $A_z$, and real numbers $a_z < b_z$, such that $z \in N_z^\circ$, $A_z = A_{\delta_z}(\theta(a_z), \theta(b_z))$ for some $\delta_z > 0$,
$$ S^1 \cap A_{z} = \{e^{i \theta(t)} : t \in [a_z,b_z] \}, \  \gamma \cap N_{z} = \gamma([a_z,b_z])$$
and $\theta(b_z) - \theta(a_z) < \pi/2$.  There are also open subsets of $\C$, say $U_z$ and $V_z$, such that $N_z \subseteq U_z$, $A_z \subseteq V_z$ and $f : U_z \to V_z$ has analytic inverse $g_z : V_z \to U_z$. Since $\gamma([\alpha, \beta])$ is compact there exist finitely many points, say $z_1, \dots , z_M$, such that
$$ \gamma([\alpha, \beta]) \subseteq \bigcup_{i=1}^M N_{z_i}^\circ $$
Let $P$ be the set formed by removing from
$$ \{a_{z_i}, b_{z_i} : i = 1, \dots, M \} \cup \{\alpha, \beta \} $$
all elements less than $\alpha$, except for the largest such one, and similarly, all elements greater than $\beta$, except for the smallest such one. The set $P$ can be labeled in increasing order so that
$$ t_0 < t_1 = \alpha < t_2 < t_3 < \dots < t_m = \beta < t_{m+1}. $$
Initially let $\delta$ be the minimum of the $\delta_{z_i}$. For $i = 0, \dots, m$ there exists $z_k$ (depending on $i$) such that $[t_i, t_{i+1}] \subseteq [a_{z_k}, b_{z_k}]$. Letting $g_i = g_{z_k}$ we put $A_{i,\delta} = A_\delta(\theta(t_i), \theta(t_{i+1}))$ and $N_{i, \delta} = g_i(A_{i, \delta})$.

It now suffices to show that by reducing $\delta$ if necessary we can ensure that the interiors of the $N_{i, \delta}$ are disjoint. Suppose there exists $i < j$ such that $N_{i, \delta}^\circ \cap N_{j, \delta}^\circ \neq \emptyset$. This implies $\theta(t_{i+1}) < \theta(t_j)$ for if $\theta(t_{i+1}) = \theta(t_j)$ then we would have  $A_{i, \delta}^\circ \cap A_{j, \delta}^\circ =\emptyset$ (because $\theta(t_{j+1}) - \theta(t_i) < \pi$) and so $N_{i, \delta}^\circ \cap N_{j, \delta}^\circ = \emptyset$, a contradiction. Thus $[t_i, t_{i+1}]$ and $[t_j, t_{j+1}]$ are disjoint and hence $\gamma([t_i, t_{i+1}])$ and $\gamma([t_j, t_{j+1}])$ are disjoint. Consequently there exist disjoint open sets $U_i$ and $U_j$ such that $\gamma([t_i, t_{i+1}]) \subseteq U_i$ and $\gamma([t_j, t_{j+1}]) \subseteq U_j$. Now reduce $\delta$ so that $g_i(A_{i, \delta}) \subseteq U_i$ and $g_j(A_{j, \delta}) \subseteq U_j$. This implies $N_{i, \delta} \cap N_{j, \delta} = \emptyset$. We continue this process of reducing $\delta$ until the interiors of the $N_{i, \delta}$ are disjoint. We then put $A_i = A_{i, \delta} = A_\delta(\theta(t_i), \theta(t_{i+1}))$ and $N_i = g_i(A_i)$ and it is clear that these sets satisfy the required properties.

By definition $S_\epsilon$ is the $\epsilon$-neighbourhood of $\gamma([\alpha, \beta])$. There exists $c > 0$ such that for each $i = 1, \dots, m-1$ we have
$$ f( N_i \cap S_\epsilon ) \supseteq A_c(\theta(t_i), \theta(t_{i+1})).$$
Therefore as $n \to \infty$
$$ \# \{ z \in N_i^\circ \cap S_\epsilon : F_n(z) = 0 \} = \frac{n(\theta(t_{i+1}) - \theta(t_i))}{2 \pi} + O(1).$$
Also, for each $\epsilon > 0$ small enough there is a corresponding $c_\epsilon > 0$ such that $c_\epsilon \to 0$ as $\epsilon \to 0$ and such that
$$ f( N_0 \cap S_\epsilon ) \subseteq A_c(\theta(t_1) - c_\epsilon, \theta(t_1))$$
and
$$ f(N_m \cap S_\epsilon) ) \subseteq A_c(\theta(t_m), \theta(t_m) + c_\epsilon) $$
Therefore as $n \to \infty$ we have for $i =0, m$
$$ K_{i, \epsilon}: =  \# \{ z \in N_i^\circ \cap S_\epsilon : F_n(z) = 0 \} \leq  \frac{n c_\epsilon}{2 \pi} + O(1)$$
Consequently,  as $n \to \infty$
\begin{eqnarray*}
 | \{ z : F_n(z) = 0 \} \cap S_\epsilon | & = & \sum_ {i=0}^m \# \{ z \in N_i^\circ \cap S_\epsilon : F_n(z) = 0 \} + O(1) \\
                                & = & K_{0, \epsilon} + K_{m, \epsilon} + \sum_{i = 1}^{m-1} \frac{n(\theta(t_{i+1}) -
                                                \theta(t_{i}))}{2 \pi} + O(1) \\
                            & = & \frac{n(\theta(\beta) - \theta(\alpha))}{2\pi} + K_{0, \epsilon} +
                            K_{1,\epsilon} + O(1)
\end{eqnarray*}
Hence
$$ \frac{\theta(\beta) - \theta(\alpha)}{2 \pi} \leq N_\epsilon \leq \frac{\theta(\beta) - \theta(\alpha)}{2 \pi} + \frac{c_\epsilon}{\pi}$$
and so
$$ \lim_{\epsilon \to 0} \ N_\epsilon = \frac{\theta(\beta) - \theta(\alpha)}{2 \pi} $$
Finally,
$$ \int_\alpha^\beta \rho(s) \  ds = \int_\alpha^\beta \frac{\theta'(s)}{2 \pi} \ ds =  \frac{\theta(\beta) - \theta(\alpha)}{2 \pi} $$
and so
$$  \lim_{\epsilon \to 0} \ N_\epsilon =  \int_\alpha^\beta \rho(s) \ ds $$
as required.
\end{proof}

The above result is sharp in the sense that one cannot let $\alpha = 0$ or $\beta = l_r$. To see this consider the function
$$F_n(z) = f_1(z)^n + f_2(z)^n + f_3(z)^n $$
where $f_1(z) = z^2(z-1)$, $f_2(z) = z-1$ and $f_3(z) = (z-1)^2$. Since $|f_1(1)| = |f_2(1)| = |f_3(1)| = 0$ we note that $1$ does not lie on any of the analytic arcs. The problem is that $F_n(z) = 0$ has $n$ solutions counted with multiplicity at $1$. If $z$ is close to $1$ and satisfies $|z| = 1$ we have
$$|f_1(z)| = |f_2(z)| > |f_3(z)| $$
and so there is an analytic arc $\gamma : (0, l) \to \C$ with endpoint $1$. Without loss of generality suppose that $\gamma(0+) = 1$. If we choose $\beta > 0$ small enough we get
$$ \int_0^\beta \rho(s) \ ds = \frac{1}{2 \pi} (\theta(\beta) - \theta(0)) < 1/2. $$
This makes use of the fact that $\theta$ can be extended continuously to $[0,l]$. For each $\epsilon > 0$ the $\epsilon$-neighbourhood $S_\epsilon$ of $\gamma( (0, \beta] )$ contains $1$ and so $N_\epsilon \geq 1$. Consequently
$$ \liminf_{\epsilon \to 0} \ N_\epsilon > \int_0^\beta \rho(s) \ ds. $$

\begin{corollary} The limit set of $\{z:F_n(z)=0\}$ as $n\to\infty$ equals the union of $L$ and those zeros of $a_r(z)$ that belong to some $B_{r,\delta}$ for some $\delta>0$.
\end{corollary}

\begin{example}
Given $c>0$, the solutions of the equation $ (z^2-1)^n=c$ are given explicitly by

$$z_r=\pm \left\{1+c^{1/n} e^{2\pi ir/n}\right\}^{1/2}$$
where $1\leq r\leq n$. Putting $f(z)=z^2-1$, the solutions converge as $n\to\infty$ to $ L=\{ z:|f(z)|=1\} $. As in Theorem \ref{structure} one may parameterize $L$ by $\gamma(u)=\pm \sqrt{1+ e^{iu}}$ where $-\pi < u < \pi$. For this initial parameterization one has $f(\gamma(u))= e^{iu}$, $\theta(u)=u$, $\rho(u)=1/(2\pi)$ and $\int_{-\pi}^\pi \rho(u)\, d u=1$. If instead one parametrizes the curve $\gamma$ by arc length using the formula
$$\frac{d s}{d u}= |\gamma'(u)|=|1+ e^{iu}|^{-1/2}$$
then one sees that the density of the zeros vanishes at $z=0$ because $s'(u) \to \infty$ as $u \to \pm \pi$.
\end{example}

\section{Spectrum of a Directed Graph}

In this section we introduce some relevant notions from graph theory and use them to study the adjacency matrix of a directed graph. As well as providing a simple introduction to the general theory, this case has some features of its own that are of interest.

Let $\cG$ denote the class of finite, directed graphs $(S,\to)$ that are irreducible in the sense that for every $u,\, v\in S$ there exists a path $u=s_1\to s_2\to ...\to s_n=v$. The irreducibility assumption implies that the indegree and outdegree of every vertex of $S$ is at least $1$.

The adjacency matrix $A$ of $(S,\to)\in\cG$ is defined by
$$
A_{i,j}=\left\{ \begin{array}{ll}%
            1&\mbox{if $i\to j$,}\\
            0&\mbox{otherwise.}
            \end{array}\right.
$$
The spectrum of $(S,\to)$ is by definition the spectrum of $A$.

We define $C$ to be the set of all $s\in S$ that have total degree $2$. If $a\in C$ and $b\to a\to e$ then irreducibility implies that $b\not= a\not= e$ if $\#(S)>1$. If $C=S$ then $S$ is a cycle and
$$\Spec(A)=\{ e^{2\pi i r/n}:r=1,2,...,n\}$$
where $n=\#(S)$. We henceforth assume that $C$ and $J=S\backslash C$ are both non-empty.

The set $C$ can be written as the union of disjoint `channels' $C_i$, which we define as subsets $T$ of $S$ that can be identified with $\{b=1,2,...,e-1,e\}$ in such a way that
\begin{enumerate}
\item every $x\in T$ has outdegree $1$; if $x<e$ then $x\to x+1$; moreover $e \to \tilde{e}\in J$;
\item every $x\in T$ has indegree $1$; if $x>1$ then $x-1\to x$; moreover $J \ni \tilde{b}\to b$.
\end{enumerate}
We write $C=\bigcup_{i=1}^h C_i$ where $C_i$ is a channel of length $e_i$ with ends $\tilde{b}_i,\, \tilde{e}_i\in J$.

The structure of $J$ may be quite complex, but every $x \in J$ has total degree at least $3$. We may write $J$ as the disjoint union of subsets $\{J_i\}_{i=1}^k$ that are connected with respect to the \emph{undirected} graph structure inherited from $(S,\to)$. It follows that there are no directed edges joining different subsets $J_i$, which we call junctions. Irreducibility implies that one can pass from any junction to another, but only via intermediate channels. In some cases one may only be able to pass between different points of the same junction via external channels.

We now lengthen each channel $C_i$ by the factor $n$ without changing its end-points $\tilde{b}_i, \, \tilde{e}_i$ or the junctions $J_i$ to produce a new graph $(S^{(n)},\to)\in\cG$,
with a corresponding adjacency matrix $A^{(n)}$. We will prove that for every $\epsilon>0$ the spectrum of $S^{(n)}$ is almost entirely confined to an annulus of the form $\{z:1-\epsilon<|z|<1+\epsilon\}$ if $n$ is sufficiently large. There may however be a small number of eigenvalues far from the unit circle. In a later section we will show that the corresponding eigenvectors are
concentrated around the set $J$. In general they are concentrated around a single junction, but this may not happen if the graph has certain symmetries.

Our general analysis depends upon collapsing each channel $C_r$ to a single point $p_r$, so that we are left with the set
$$
\tilde{S}= J\cup\{p_r:1\leq r\leq h\}.
$$
If $i,\, j\in J$ then we write $i\to j$ in $\tilde{S}$ if and only if $i\to j$ in $S$; we also include the extra edges $\tilde{b}_r\to p_r$ and $p_r\to \tilde{e}_r$. $\tilde{A}$ is defined to be the adjacency matrix of $(\tilde{S},\to)\in\cG$. Note that $(\tilde{S},\to)$ and $\tilde{A}$ depend neither on $n$ nor on the lengths $e_r$ of the original channels.

\begin{theorem}\label{polypencil} The eigenvalues of $A^{(n)}$ coincide with the roots of the polynomial pencil $D(z)-\tilde{A}$ on $\tilde{S}$, where $D(z)$ is the diagonal matrix
with entries
$$
D_{i,j}(z)=\left\{ \begin{array}{ll}
        z       &\mbox{if $i=j\in J$,}\\
        z^{ne_r}    &\mbox{if $i=j=p_r$ for some $r\in \{1,\dots,h\}$,}\\
        0       &\mbox{otherwise.}
\end{array}\right.
$$
Indeed
\begin{equation}
\det(zI-A^{(n)})=\det(D(z)-\tilde{A})\label{detidentity1}
\end{equation}
for all $z\in\C$.
\end{theorem}

\begin{proof} It suffices to prove the following more general result which will also be of use later on.

Let $M$ be a matrix of the form
$$ M = \left( \begin{array}{ccccc}
            C_1         &       &       &       & E_1   \\
                    & C_2   &       &       & E_2   \\
                    &       & \ddots    &       & \vdots    \\
                    &       &       & C_h   & E_h   \\
            B_1     & B_2   & \cdots    & B_h   & J
        \end{array}
    \right )
$$
where $C_k$ is the $n_k \times n_k$ matrix
$$ \left ( \begin{array} {cccc}
            \alpha_k    & -1        &       &       \\
                    & \alpha_k  & \ddots        &       \\
                    &       & \ddots    & -1        \\
                    &       &       & \alpha_k
        \end{array}
    \right )
$$
with $\alpha_k$ on the diagonal, $-1$ directly above the diagonal and $0$ elsewhere; $J$ is an arbitrary $m \times m$ matrix; $B_k$ is an $m \times n_k$ matrix with 0 everywhere except possibly for entry $(i_k, 1)$ where it has the value $\beta_k$; and $E_k$ is an $n_k \times m$ matrix with $0$ everywhere except possibly for entry $(n_k, j_k)$ where it has the value $\epsilon_k$. It is clear that the vertices of the graph $(S^{(n)}, \to)$ can be labeled so that the corresponding matrix $zI - A^{(n)}$ has the above form.

On the set of matrices of the above form we may define a map $M \mapsto \tilde{M}$ by defining the corresponding matrix $\tilde{M}$ by
$$ \tilde{M} = \left ( \begin{array}{ccccc}
            \alpha_1^{n_1}      &           &           &           & E'_1\\
                            & \alpha_2^{n_2}&           &           & E'_2\\
                            &           & \ddots        &            & \vdots \\                                        &           &           & \alpha_h^{n_h} & E'_h \\
            B'_1            & B'_2  & \cdots        & B'_h      & J
            \end{array}
        \right )
$$
where $B'_k$ is the $m \times 1$ column vector with value $\beta_k$ in position $i_k$ and $0$ elsewhere; $E'_k$ is the $1 \times m$ row vectors with value $\epsilon_k$ in position $j_k$ and $0$ elsewhere;  and where the $1\times 1$ matrix $(a_k^{n_k})$ replaces the matrix $C_k$.  It is clear that the vertices of the graph $(\tilde{S}, \to)$ can be labeled so that the corresponding matrix $D(z) - \tilde{A}$ has the above form.

It now suffices to show that $\det(M) = \det(\tilde{M})$. We prove the result by induction on $h$ which is the number of matrices $C_k$ along the diagonal of $M$. The result clearly holds for $h=0$ and for any size matrix $J$. For the inductive case, let $N = n_1 + \dots + n_h$ and let $K$ be the matrix
$$ K = \left( \begin{array}{cccc}
            C_2 &       &       & E_2   \\
                & \ddots    &       & \vdots    \\
                &       & C_h   & E_h   \\
             B_2    & \cdots    & B_h   & J
        \end{array}
    \right )
$$
Expanding the determinant of $M$ down the first column gives
\begin{eqnarray*}
\det(M) & = & \alpha_1 \det(M_{1,1}) + (-1)^{N+ i_1 +1} \beta_1 \det(M_{N+i_1,1}) \\
        & = &  \alpha_1^{n_1} \det(K) + (-1)^{N + i_1 + 1}\beta_1(-1)^{n_1 - 1}
            (-1)^{N - n_1 + j_1 + 1} \epsilon_1 \det(K_{N - n_1 + i_1, N - n_1 + j_1}) \\
        & = & \alpha_1^{n_1} \det(K) + (-1)^{i_1 + j_1 + 1} \beta_1 \epsilon_1 \det(K_{N - n_1 + i_1, N - n_1 + j_1})
\end{eqnarray*}
where $\det(M_{1,1})$ is evaluated by expanding successive minors down the first column and $\det(M_{N+i_1, 1})$  is evaluated by expanding successive minors across the first row. Similarly,
\begin{eqnarray*}
\det(\tilde{M}) & =  & \alpha_1^{n_1} \det(\tilde{K}) + (-1)^{h+i_1+1} \beta_1 (-1)^{h-1+j_1+1}\epsilon_1
                \det( \tilde{K}_{h -1+i_1, h -1 + j_1})  \\
             & = & \alpha_1^{n_1} \det(\tilde{K}) + (-1)^{i_1 + j_1 +1} \beta_1 \epsilon_1 \det(\tilde{K}_{h -1+i_1, h -1 + j_1})
\end{eqnarray*}
By the induction hypothesis
$$ \det(K) = \det( \tilde{K} ) $$
$$ \det( K_{N - n_1 +i_1, N - n_1 + j_1} ) = \det( \tilde{K}_{h -1+i_1, h -1 + j_1}) $$
and so $\det(M) = \det( \tilde{M})$, as required.
\end{proof}

\begin{corollary} There exists a polynomial $p$ of the form
$$p(z,w)=a_c(z)w^c+a_{c-1}(z)w^{c-1}+\dots+a_1(z)w+a_0(z)$$
such that
\begin{equation}
\det(zI-A^{(n)})=p(z,z^n).\label{detidentity2}
\end{equation}
Moreover $c=\#(C)$. All of the polynomials $a_r$ have
degree at most $\#(J)$ and $a_c$ has degree equal to
$\#(J)$.
\end{corollary}

\begin{proof} Theorem \ref{polypencil} implies $\det(z - A^{(n)})  = \det(D(z) - \tilde{A})$, and after the change of variable $w = z^n$ we note that
$$ \det(D(z) - \tilde{A})  = \left| \begin{array}{ c c c c c}
                    w^{e_1} &       &       &       & E'_1  \\
                            & w^{e_2} &     &       & E'_2  \\
                            &       & \ddots    &       & \vdots    \\
                            &       &       & w^{e_h}   & E'_h  \\
                    B'_1        & B'_2  & \dots & B'_h  & zI-J
                    \end{array}
                \right|
$$
Using the permutation definition of the determinant we see that the highest order power of $w$ corresponds to the term in the determinant that includes all of the terms $w^{e_r}$, and the power of $w$ is then $e_1 + \dots + e_h = c$. The term in the determinant arising from the identity permutation shows that $a_c(z)$ has degree equal to $\#(J)$. It is clear that this is the maximum degree of the polynomials $a_r(z)$.
\end{proof}

It is easy to construct numerical examples which illustrate the asymptotic spectral behaviour proved in Lemma~\ref{asympcircle}, but we give a number of more general examples below.

\section{The General Model}

We now turn to the general model described in the introduction. Given a matrix $A$ we construct the associated graph and define the channels and junctions as described before, except for the following variations. We allow every point $x$ in a channel to have an extra edge which starts and ends at that point provided the diagonal matrix entry $A_{x,x}$ is the same for all points in the channel. We also require that all the non-zero matrix entries $A_{x,y}$ associated with a particular channel, including the two that link it to external junctions, should be equal. Thus every channel $C_i$ is determined by three parameters, its length $e_i$, the diagonal entries $\alpha_i$, which may or may not vanish, and the non-zero, off-diagonal entries $\beta_i$ associated with edges that link successive points in the channel.

Given a positive integer $n$ we now construct a new matrix $A^{(n)}$ which has the same graph structure and matrix entries as $A$, except that the length of each channel is $ne_i$ instead of $e_i$. We are interested in the asymptotic behaviour of the eigenvalues of $A^{(n)}$ as $n\to\infty$. The method of Theorem~\ref{polypencil} implies that
\begin{equation}
p_n(z):=\det(zI-A^{(n)})=k\sum_s a_s(z)\prod_{i\in s} ((z-\alpha_i)/\beta_i)^{ne_i}\label{pgeneral}
\end{equation}
where $k=\prod_{i=1}^h \beta_i^{ne_i}$, the sum is over all subsets $s$ of $\{1,2,\dots,h\}$ and each $a_i$ is a polynomial of degree at most $d=\sum_{j=1}^k \#(J_j)$. If $\sigma= \{1,2,\dots,h\}$ then $a_\sigma$ is of degree $d$. The other $a_s$ may be of lower degree and may vanish identically. The zeros of $p_n(z)$ coincide with those of
$$ F_n(z) = \sum_s a_s(z) f_s(z)^n $$
where
$$f_s(z) = \prod_{i\in s} ((z-\alpha_i)/\beta_i)^{e_i} $$
and so the spectral asymptotics are determined by Theorem~\ref{mostgeneral}.

Since there are $2^h$ subsets of $\{1, 2, \dots, h\}$ one would expect the number of terms in $p_n(z)$ to grow exponentially as the number of channels is increased. Surprisingly this often does not occur as many of the $a_s$ may vanish identically. We will prove this for the case when each junction consists of a single point. We will need the following coefficient theorem for directed graphs; see \cite[Theorem~1.2]{sachs}. Because of the presence of the term $zI$ we allow $i\to i$ to be an edge for all $i\in S$ whether or not it was previously. A single point of $S$ is regarded as a cycle of length $1$.

\begin{theorem}\label{cycles}
If $s \subseteq \{1,2,\dots, h\}$ and $a_s$ is non-zero then after removing every channel $C_i$ such that $i\in s$ from the graph it is possible to cover the remainder by disjoint cycles such that each channel $C$ not in $s$ is completely contained in one of the cycles.
\end{theorem}

\begin{proof} We use the formula
\begin{equation}
p(z)=\sum_{\pi}\sign(\pi) \prod_{i\in
S}A_{z,i,\pi(i)}\label{permdet}
\end{equation}
where $A_z=zI-A$ and the sum is over all permutations $\pi$. We note that the term $t_\pi =\prod_{i\in S}A_{z,i,\pi(i)}$ is non zero if and only if  $(i,\pi(i))$ is a directed edge of $S$ for every $i$. Every $\pi$ may be written as a product of disjoint cycles, some of which may be of length $1$. If one point of a channel $C$ is a cycle of length $1$ then the same applies to every point of $C$ and $t_\pi$ contains the factor $((z-\alpha_i)/\beta_i)^{n_i}$. Otherwise the contribution of the channel $C$ to $t_\pi$ is a non-zero constant. The polynomial $a_s$ is therefore obtained by summing over those partitions of $S$ into disjoint cycles such that every point of every channel in $s$ is a singleton cycle, and each channel not in $s$ is a part of a cycle of length greater than $1$. Since $a_s$ is non-zero such a partition actually exists.
\end{proof}

\begin{corollary}\label{balance} Suppose each junction in $S$ consists of a single point. Let $d$ be the number of junctions and let $h$ be the number of channels. If $N$ is the number of subsets s of $\{1, \dots, h \}$ such that $a_s$ is non zero then $N = O(h^d)$ as $h\to \infty$.
\end{corollary}

\begin{proof} Let $G_s$ denote the graph obtained from $S$ by removing each channel $C_i$ for $i\in s$. If $G_s$ can be covered by disjoint cycles such that each channel $C$ not in $s$ is contained in a single cycle then for each junction in $G_s$ we either have precisely one in channel and one out channel, or else we have no channels at all. So we are now left with a counting problem. The number of subsets $s$ such that \emph{every} junction in $G_s$ has precisely one in channel and one out channel is at most $\binom{h}{d}$ since each possibility requires precisely $d$ channels, and we have $h$ from which to choose. Similarly, the number of subsets $s$ such that $k$ junctions in $G_s$ have precisely one in channel and one out channel with the remaining junctions having no channels is at most $\binom{h}{k}$. Consequently,
$$N \leq \binom{h}{0} + \binom{h}{1} + \dots + \binom{h}{d} = O(h^d) $$
as $h \to \infty$, as required.
\end{proof}

Despite the above result, there are examples for which $a_s$ is non zero for all $2^h$ subsets $s$ of $\{1,\dots, h\}$. Let $S$ be a graph consisting of one junction and $h$ channels that satisfy the following properties. The junction $J$ is a cycle of length $2h$ with
$$1 \to 2 \to 3 \to \dots \to 2h-1 \to 2h \to 1$$
For $i =1, \dots, h$ the channel $C_i$ begins at $i$ and ends at $i+1$ so that
$$ i \to C_i \to i+1$$
The length of each channel is not important. If $s$ is a subset of $\{1, \dots, h \}$ then the only way of covering the graph $G_s$ by disjoint cycles so that each channel $C$ not in $s$ is contained in a single cycle is by using the unique cycle in $G_s$ that contains all the channels. Therefore each $a_s$ is actually a non zero constant.

\section{Some Examples}

The matrices $A^{(n)}$ considered in this paper may be classified according to the number of channels $h$ and the number of junctions $k$. In this section we consider some of the phenomena that arise for small values of $h$ and $k$, our goal being to describe the asymptotic form of $\Spec(A^{(n)})$ as $n\to\infty$. Since each junction has at least one in channel and one out channel by irreducibility, it follows that $k\leq h$. Even if we assume for simplicity that each junction consists of a single point and that each channel has the same length $n$, there are still several graphs for each choice of $h$ and $k$.

\begin{example}[$h=k=1$] We suppose that the graph has one junction containing two points and one channel containing $n-2$ points. The $n\times n$ matrix is taken to have the form
$$
A=\left\{ \begin{array}{ll}%
1&\mbox{if $j=i+1$,}\\
a&\mbox{if $i=j=1$,}\\
b&\mbox{if $i=1$ and $j=n$,}\\
c&\mbox{if $j=1$ and $i=n$,}\\
d&\mbox{if $i=j=n$,}\\
0&\mbox{otherwise.}
\end{array}\right.
$$
in which the constants $a,b,c,d$ describe boundary conditions at the ends of the ordered interval $\{1,2,...,n\}$. One may check that the characteristic polynomial of $A$ is
$$
p_n(z)=z^{n-2}\left( z^2-(a+d)z+ad-bc\right) - c.
$$
The roots of this polynomial converge asymptotically to the unit circle, with the exception of isolated roots that converge to any solution of $z^2-(a+d)z+ad-bc=0$ that satisfies $|z|>1$.
\end{example}

\begin{example}[$h=2$, $k=1$] We consider the case in which the graph associated to $A$ has only one junction and that junction contains only one point. We also assume that there are $h=2$ channels, each of which has length $n$. The matrix $A$ therefore has $2h+1=5$ free parameters apart from $n$. We will see that all the anti-Stokes curves are circles, where we regard straight lines as circles with infinite radii. For $n=3$ the matrix $A$ is of the form
$$
A = \left( \begin{array}{ccc|ccc|c}
        \alpha_1    & \beta_1       &       &           &       &       &    \\
                & \alpha_1  & \beta_1   &           &       &       &   \\
                &           & \alpha_1&             &       &       & \beta_1 \\ \hline             &           &        & \alpha_2     & \beta_2 &     &       \\
                &           &       &           & \alpha_2 & \beta_2 &       \\
                &           &       &           &       & \alpha_2 & \beta_2 \\ \hline
        \beta_1     &           &       & \beta_2       &       &        & \gamma
        \end{array}
    \right)
$$
where we need to assume that $\beta_1$ and $\beta_2$ are non-zero. For general $n$ the reduced matrix is of the form
$$
A(z) = \left( \begin{array}{ccc}
        (z-\alpha_1)^n      &               &   - \beta_1^n     \\
                        & (z-\alpha_2)^n    &   -\beta_2^n  \\
        \beta_1         & \beta_2           &   z-\gamma
        \end{array}
    \right )
$$
and the characteristic polynomial is
$$
p(z)=(z-\gamma)(z-\alpha_1)^n(z-\alpha_2)^n-\beta_1^{n+1}(z-\alpha_2)^n
-\beta_2^{n+1}(z-\alpha_1)^n.
$$
This may be simplified further if $\alpha_1=\alpha_2$ and we assume that this is not the case.

The characteristic polynomial is of the form $\sum_r a_r(z) f_r(z)^n$ where
\begin{eqnarray*}
f_1(z)&=& (z-\alpha_1)(z-\alpha_2)\\
f_2(z)&=& \beta_1(z-\alpha_2)\\
f_3(z)&=& \beta_2(z-\alpha_1)
\end{eqnarray*}
and $a_1(z) = z-\gamma$, $a_2(z) = -\beta_1$, $a_3(z) = -\beta_2$. Each of the functions $f_i$ dominates the others in absolute value in an open set $U_i$, where $U_1$ contains all large enough $z$, $\alpha_1\in U_2$ and $\alpha_2\in U_3$. The limit set $E$ of the spectrum of $A$ as $n\to\infty$ is contained in the union of the anti-Stokes lines, with the possible exception of an eigenvalue that converges to $\gamma$. Whether or not this eigenvalue exists depends on the parameters of the matrix.

The anti-Stokes lines are defined for $i\neq j$ by
$$
K_{i,j}=\{z:|f_i(z)|=|f_j(z)|\}
$$
and are given by the formulae
\begin{eqnarray*}
K_{1,2}&=& \{z:|z-\alpha_1|=|\beta_1|\}\\
K_{1,3}&=& \{z:|z-\alpha_2|=|\beta_2|\}\\
K_{2,3}&=& \{z:|\beta_1|\,|z-\alpha_2|=|\beta_2|\,|z-\alpha_1|\}
\end{eqnarray*}
An elementary calculation shows that $K_{2,3}$ is a circle with centre $(\delta\alpha_2-\alpha_1)/(\delta-1)$, where $\delta=|\beta_1/\beta_2|^2$. If $\delta=1$ then $K_{2,3}$ is a straight line. The three circles are part of a coaxial system.

Once one has determined the circles $K_{i,j}$ the regions $U_i$ can be identified by use of the following rules. Each component of $U=\C\backslash (K_{1,2}\cup K_{1,3}\cup K_{2,3})$ is contained in a single set $U_i$ and the unbounded components are contained in $U_1$. If $C$ is some part of $K_{i,j}$ and $U_i$ is on one side of $C$ then by applying the maximum principle to $f_i(z)/f_j(z)$ one sees that $U_j$ is on the other side of $C$. If, however, $U_k$ is on one side of $C$ where $k\neq i$ and $k\neq j$ then $U_k$ is also on the other side of $C$ and the asymptotic spectrum $E$ does not contain any points in $C$.

By applying the above rules one can progressively determine which $U_i$ contains each component of $U$ and also eliminate certain curves $C\subseteq K_{i,j}$ as possible parts of $E$.  Figure~\ref{fig:h2k1}  portrays the eigenvalues of the matrix $A$ with $n=30$, $\alpha=[2, -1]$,
$\beta=[2,3]$ and $\gamma=5$. The circles involved are $|z-2|=2$, $|z+1|=3$ and $|z-4.4|=3.6$. The arcs removed from each circle are in accordance with the application of the above rules. The eigenvalue close to $5$ is given more accurately by $\lambda \sim 5.0000104$.
\end{example}

\begin{figure}[h!]
\centering
\scalebox{0.7}{\includegraphics{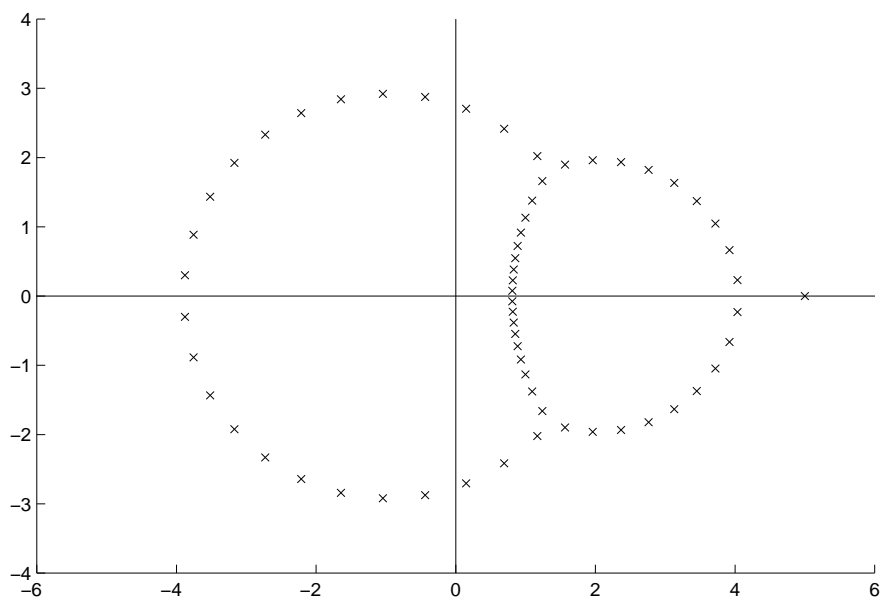}}
\caption{Eigenvalues of a matrix with $h=2$ and $k=1$.}
\label{fig:h2k1}
\end{figure}

If the two channels have different lengths, say $c_1n$ and $c_2n$, where $n\to\infty$, then only small changes are needed. The anti-Stokes curves $K_{1,2}$ and $K_{1,3}$ are unchanged but $K_{2,3}$ becomes the set of $z$ such that
$$
\left|\frac{z-\alpha_1}{\beta_1}\right|^{c_1}=%
\left|\frac{z-\alpha_2}{\beta_2}\right|^{c_2},
$$
which is only a circle if $c_1=c_2$. The general rules for locating
the eigenvalues still apply.

\begin{example}[$h=3$, $k=1$] The analysis in the previous example carries through for larger values of $h$.  In particular for $h=3$ the reduced matrix has $7$ free parameters apart from $n$ and is
$$
A(z) = \left ( \begin{array}{cccc}
            (z-\alpha_1)^n      &               &               &    -\beta_1^n \\
                            & (z-\alpha_2)^n    &               &    -\beta_2^n \\
                            &               & (z-\alpha_3)^n    &    -\beta_3^n \\
            -\beta_1            & -\beta_2          & -\beta_3          &    z-\gamma
        \end{array}
    \right )
$$
The characteristic polynomial is
\begin{eqnarray*}
p(z)&=&(z-\gamma)(z-\alpha_1)^n(z-\alpha_2)^n(z-\alpha_3)^n%
-\beta_1^{n+1}(z-\alpha_2)^n(z-\alpha_3)^n\\%
&& \ \  - \beta_2^{n+1}(z-\alpha_1)^n(z-\alpha_3)^n%
-\beta_3^{n+1}(z-\alpha_1)^n(z-\alpha_2)^n.
\end{eqnarray*}

Assuming that all the $\alpha_i$ are distinct, there are $6$ anti-Stokes lines, namely
\begin{eqnarray*}
K_{1,2} &=& \{ z:|z-\alpha_1|=|\beta_1|\}\\
K_{1,3} &=& \{ z:|z-\alpha_2|=|\beta_2|\}\\
K_{1,4} &=& \{ z:|z-\alpha_3|=|\beta_3|\}\\
K_{2,3} &=& \{ z:|\beta_2|\,|z-\alpha_1|=|\beta_1|\,|z-\alpha_2|\}\\
K_{2,4} &=& \{ z:|\beta_3|\,|z-\alpha_1|=|\beta_1|\,|z-\alpha_3|\}\\
K_{3,4} &=& \{ z:|\beta_2|\,|z-\alpha_3|=|\beta_3|\,|z-\alpha_2|\}
\end{eqnarray*}
These are all circles, or straight lines in degenerate cases.

In Figure~\ref{fig:h3k1}, we put $n=30$, $\alpha=[1,i,-i]$, $\beta=[1,3/2,3/2]$ and $\gamma=3$. One again the eigenvalues are close to the parts of the anti-Stokes lines obtained by using the rules given above.
\end{example}

\begin{figure}[h!]
\centering
\scalebox{0.7}{\includegraphics{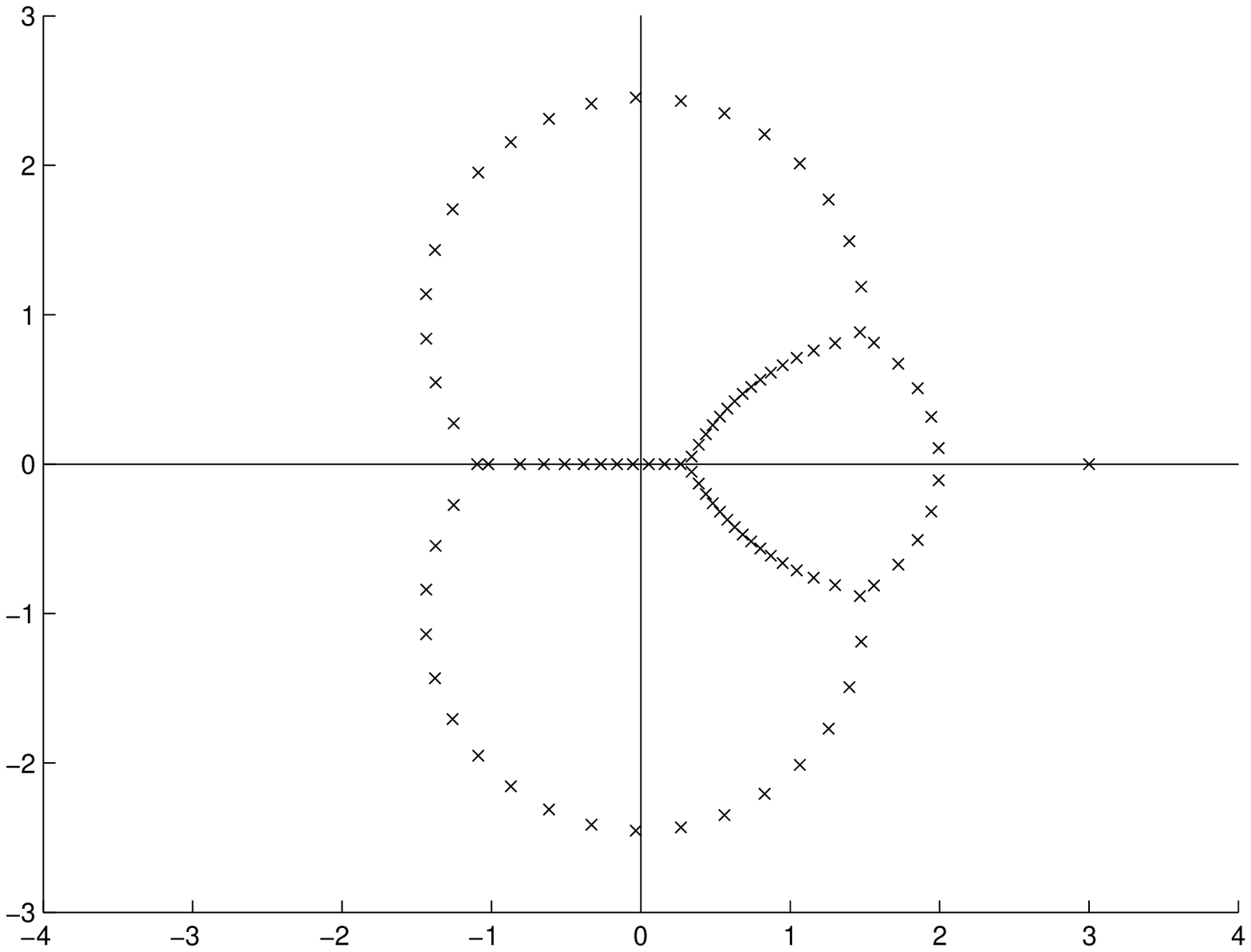}}
\caption{Eigenvalues of a matrix with $h=3$ and $k=1$.}
\label{fig:h3k1}
\end{figure}

\begin{example}[$h=k=2$] We consider a graph with two junctions, each containing one point, and two channels, each containing $n$ points. Each channel goes from one junction to the other, so that the graph has exactly one circuit. The matrix $A$ acts on vectors of length $2n+2$. The indices $2n+1,\, 2n+2$ label the junctions, $\{1,\dots,n\}$ label the first channel and $\{n+1,\dots,2n\}$ label the second channel. The matrix $A$ is defined by
$$
A=\left ( \begin{array}{cccc| cccc| cc}
        \alpha_1     & \beta_1  &       &       &           &       &        &  &   &   \\
                & \alpha_1  & \ddots    &       &           &       &        &  &   & \\
                &           & \ddots    & \beta_1   &           &       &        &  &   & \\
                &           &       & \alpha_1&         &       &       &    &  & \beta_1 \\ \hline
                &           &       &       & \alpha_2  & \beta_2   &        &  &   &   \\
                &           &       &       &           & \alpha_2& \ddots   &  &   &   \\
                &           &       &       &           &       & \ddots     & \beta_2  &   &   \\
                &           &       &       &           &       &       & \alpha_2 & \beta_2& \\ \hline
        \beta_1 &           &       &       &           &       &       &        & \gamma_1 &   \\
                &           &       &       & \beta_2       &       &        &      &   &\gamma_2
        \end{array}
    \right )
$$
The characteristic polynomial for general $n$ is
$$
p(z)=(z-\gamma_1)(z-\gamma_2)(z-\alpha_1)^n(z-\alpha_2)^n-\beta_1^{n+1}\beta_2^{n+1}.
$$
The spectrum converges to
\begin{equation}
\{ z:|(z-\alpha_1)(z-\alpha_2)|=|\beta_1\beta_2|\}\,
,\label{h2k2}
\end{equation}
as $n\to\infty$, apart from the possibility of isolated eigenvalues converging to $\gamma_1$ or $\gamma_2$. This curve has two points in common with each of the circles $|z-\alpha_i|=|\beta_i|$, provided these circles intersect.

We carried out computations for the case $n=30$, $\alpha=[-1.2,1.2]$, $\beta=[1.3,1.3]$ and $\gamma=[-2,2]$. Figure~\ref{fig:h2k2ps} plots the eigenvalues as crosses and the circles $|z-\alpha_i|=|\beta_i|$ as dotted curves. The curve (\ref{h2k2}) is simple but non-convex and crosses both circles at $\pm 0.5i$.

The other curves in the figure are the pseudospectral contours of $A$ for $\epsilon=10^{-m}$, where $m=0,...,6$, the outermost one, for $\epsilon=1$, being only partly visible; see \cite{LOTS,TE} for the definition. One deduces from these contours that the eigenvalues are highly unstable; indeed for $n=50$ they are not easily computable because of rounding errors. Note that the pseudospectral contours are related much more strongly to the pair of circles than they are to the eigenvalues.

If one links the two junctions together weakly by putting $$A(2n+1,2n+2)=A(2n+2,2n+1)=10^{-2}$$ then the spectrum of $A$ changes radically and approximates the union of the two circles more closely as $n$ increases. The case $n=100$ is plotted in Figure~\ref{fig:h2k2-100}.
\end{example}

\begin{figure}[h!]
\centering
\scalebox{0.75}{\includegraphics{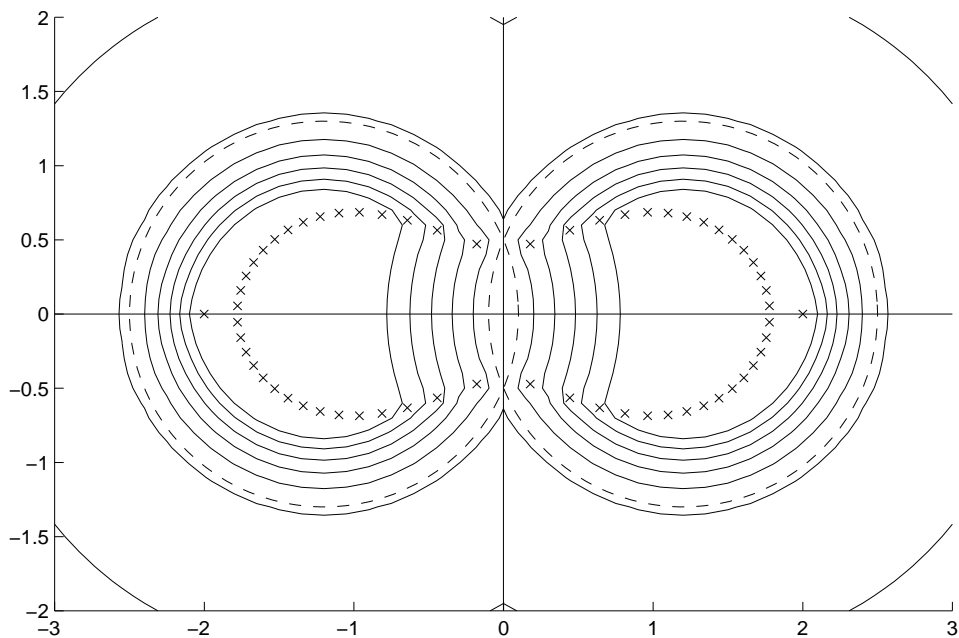}}
\caption{Eigenvalues of a matrix with $h=2$, $k=2$ and $n=30$.}
\label{fig:h2k2ps}
\end{figure}

\begin{figure}[h!]
\centering
\scalebox{0.75}{\includegraphics{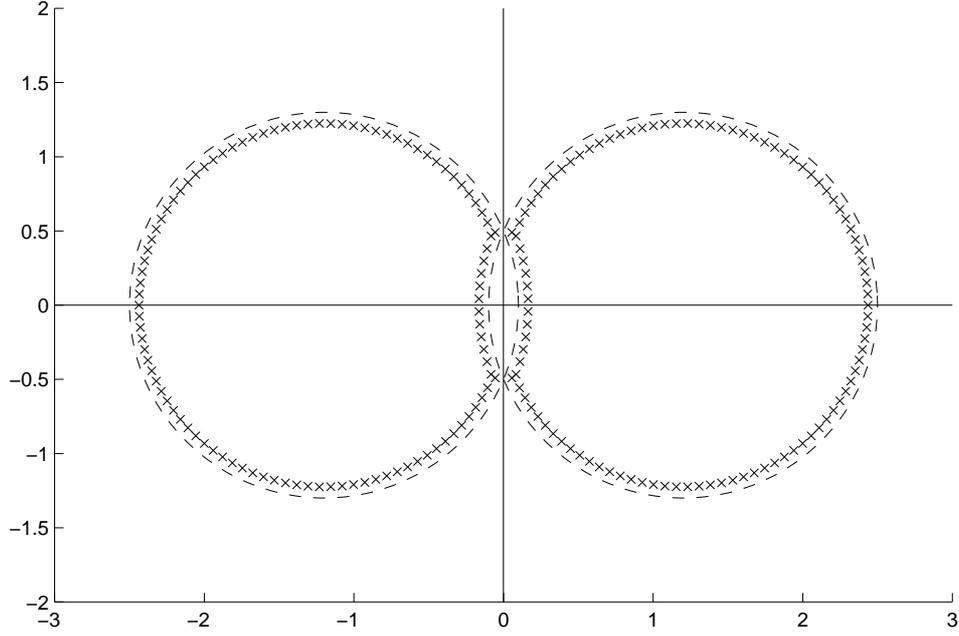}}
\caption{Eigenvalues of a matrix with $h=2$, $k=1$ and $n=100$.}
\label{fig:h2k2-100}
\end{figure}

\begin{example}[$h=3$, $k=2$] \label{three} The $(3n+3)\times (3n+3)$ matrix
$$
A_{i,j}=\left\{\begin{array}{ll}%
i/2&\mbox{if $1\leq i=j\leq n$,}\\
-2&\mbox{if $i=j=n+1$,}\\
-i/2&\mbox{if $n+2\leq i=j\leq 2n+1$,}\\
2&\mbox{if $i=j=2n+2$,}\\
3/2&\mbox{if $2n+3\leq i=j\leq 3n+2$,}\\
0&\mbox{if $i=j=3n+3$,}\\
1&\mbox{if $i+1=j$,}\\
1&\mbox{if $i=3n+3$, $j=1$,}\\
1&\mbox{if $i=n+1$, $j=3n+3$,}\\
0&\mbox{otherwise.}
\end{array}\right.
$$
has characteristic polynomial
\begin{eqnarray*}
\det(zI-A)&=& z(z^2 - 4) (z - i/2)^n(z + i/2)^n (z - 3/2)^n - (z -2)(z+i/2)^n(z-3/2)^n -1
\end{eqnarray*}
Its spectrum is depicted in Figure~\ref{fig:three} for $n=25$. The corresponding graph has $2$ junctions and $3$ channels, each of length $n$. Since $h=3$ there are potentially $8$ terms in (\ref{pgeneral}), but the explicit expression of $\det(zI - A)$ shows that only $3$ are non-zero. These give rise to the $3$ arcs in Figure~\ref{fig:three}. The eigenvalues near $\pm 2$ are associated with two of the three diagonal entries at the junctions.
\end{example}

\begin{figure}[h!]
\centering
\scalebox{0.7}{\includegraphics{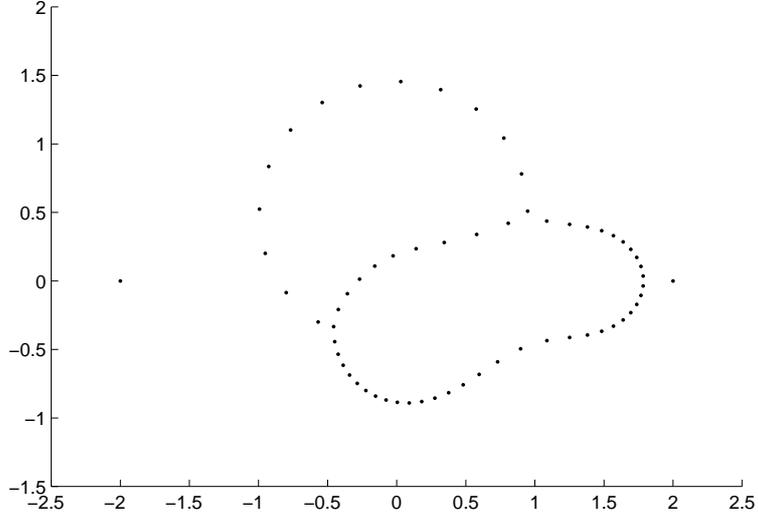}}
\caption{The spectrum of Example~\ref{three}.}
\label{fig:three}
\end{figure}

\section{Localization of Eigenvectors}

In the general model each channel $C_r$ is determined by three parameters, its length $e_r$, the diagonal entries $\alpha_r$ and the non-zero, off-diagonal entries $\beta_r$. We may associate with each channel $C_r$ the circle
$$ S_r = \{ z : |z - \alpha_r| = |\beta_r| \}.$$
If $v^{(n)}$ is an eigenvector of $A^{(n)}$ corresponding to the eigenvalue $\lambda^{(n)}$, it will be shown that the coordinates of $v^{(n)}$ corresponding to the channel $C^{(n)}_r$ are influenced by the proximity of $\lambda^{(n)}$ to the circle $S_r$. For convenience we label the vertices of channel $C^{(n)}_r$ by $\{1, \dots, n e_r \}$ so that the coordinates of $v^{(n)}$ corresponding to $C^{(n)}_r$ are $v^{(n)}_1, \dots, v^{(n)}_{n e_r}$.

\begin{theorem}\label{localize} Let $\{\lambda^{(n)}\}$ be a sequence such that $\lambda^{(n)}$ is an eigenvalue of $A^{(n)}$, and let $v^{(n)}$ denote any eigenvector corresponding to $\lambda^{(n)}$.
\begin{enumerate}
\item If there exists $\delta > 0$ such that
\begin{equation}
\dist(\lambda^{(n)}, S_r) \geq \delta \label{cond1}
\end{equation}
for all $n$, then there exists $c \in (0,1)$ such that for each $n$ either
$$ |v^{(n)}_i| \leq c^{i-1}|v^{(n)}_1| \ \ {\rm for} \ \ i = 1, \dots, ne_r $$
or
$$ |v^{(n)}_{ne_r-i}| \leq c^{i}|v^{(n)}_{ne_r}| \ \ {\rm for} \ \ i = 0, \dots, ne_r-1. $$

\item If there exists $C > 0$ such that
\begin{equation} \label{cond2}
\dist(\lambda^{(n)}, S_r) \leq \frac{C}{n}
\end{equation}
for all $n$, then there exists $d > 0$ such that for each $n$ either
$$ d^{-1} \leq \left | \frac{v^{(n)}_i}{v^{(n)}_j} \right | \leq d \ \ {\rm for \ all} \ \ i,j \in \{ 1, \dots, n e_r \}$$
or
$$ v^{(n)}_i = 0 \ \ {\rm for} \ \ i = 1, \dots, ne_r . $$
\end{enumerate}
\end{theorem}

\begin{proof} The equation
$$ (\lambda^{(n)} I - A^{(n)}) v^{(n)} = 0 $$
implies
$$ (\lambda^{(n)} - \alpha_r) v^{(n)}_i = \beta_r v^{(n)}_{i+1} \ \ {\rm for} \ \ i = 1, \dots, ne_r - 1. $$
If $v^{(n)}_1 = 0$ then $v^{(n)}_i = 0$ for $i = 1, \dots, ne_r$. So suppose $v^{(n)}_1 \neq 0$.
Condition (\ref{cond1}) implies there exists a positive constant $c < 1$ such that either $|(\lambda^{(n)}-\alpha_r) / \beta_r| < c$ or $ |\beta_r/ (\lambda^{(n)} - \alpha_r)| < c$ for each $n$.  If $| (\lambda^{(n)} - \alpha_r)/\beta_r  | < c$ then
$$ |v_i^{(n)}| \leq c^{i-1} |v^{(n)}_1|  \ \ {\rm for} \ \ i = 1, \dots, ne_r $$
and if $|\beta_r/(\lambda^{(n)} - \alpha_r)| < c$ then
$$ |v^{(n)}_{n e_r - i}| \leq c^{i} |v^{(n)}_{n e_r}| \ \ {\rm for} \ \ i = 0, \dots, ne_r-1 . $$
and this proves 1). To prove 2) we observe that condition (\ref{cond2}) implies there exists $a>0$ such that
$$ 1 - \frac{a}{n} \leq \left| \frac{\lambda^{(n)} - \alpha_r}{\beta_r} \right| \leq 1 + \frac{a}{n} $$
and so
$$ \left( 1 - \frac{a}{n} \right)^{i-1}|v^{(n)}_1| \leq |v^{(n)}_i| \leq   \left( 1 + \frac{a}{n} \right)^{i-1}|v^{(n)}_1| $$
for $i = 1, \dots , n e_r$. Therefore
$$   \frac{(1-\frac{a}{n})^n}{(1+\frac{a}{n})^n} \leq \left | \frac{v^{(n)}_i}{v^{(n)}_j} \right | \leq
 \frac{(1+\frac{a}{n})^n}{(1-\frac{a}{n})^n}  $$
for all $i, j \in \{ 1, \dots, n e_r \}$.  Since $(1 \pm \frac{a}{n})^n \to e^{\pm a}$ as $n \to \infty$ the result follows.
\end{proof}

If the eigenvalues $\lambda^{(n)}$ are a positive distance away from all the circles $S_i$ then the corresponding eigenvectors $v^{(n)}$ decrease exponentially along all the channels. This suggests that in this case the eigenvectors will be concentrated around the junctions. We make this precise as follows.

Given $n$ and $N$ we define the subset $C_{n,N}$ of the graph $(S^{(n)},\to)$ by
$$
C_{n,N}=\bigcup_{r=1}^h\{ i\in C^{(n)}_r:N\leq i\leq ne_r-N\}.
$$
We say that a sequence of normalized eigenvectors $v^{(n)}$ are localized around the junctions if for all $\epsilon > 0$ there exists $N$, depending on $\epsilon$, such that $\Vert v^{(n)} |_{C_{n,N}} \Vert_2< \epsilon$ for all $n$. Note that if
$$
n\max\{e_i:1\leq i\leq h\}< 2N
$$
then $C_{n,N}=\emptyset$, so the bound is automatic. The localization condition therefore refers to the asymptotic behaviour of $A^{(n)}$ as $n\to\infty$.

\begin{theorem}\label{uniformlocalize} Let $\{v^{(n)}\}$ be a sequence of normalized eigenvectors of $A^{(n)}$ with corresponding eigenvalues $\lambda^{(n)}$. If there exists $\delta > 0$ such that for 
\[
\dist( \lambda^{(n)}, S_r) \geq \delta %
\hspace{2em} \mbox{for all $n$ and all $r = 1, \dots, h$}
\]
then the $v^{(n)}$ are localized around the junctions.
\end{theorem}

\begin{proof} Let $\epsilon>0$ and choose $N$ large enough so that
$$\frac{c^{2(N-1)}}{1 - c^2} < \frac{\epsilon^2}{h}.$$
By Theorem~\ref{localize} there exists a positive constant $c < 1$ such that for $r = 1,\dots, h$ and all $n$ either
$$ |v^{(n)}_i| \leq c^{i-1} \ \ {\rm for} \ \ i = 1, \dots, ne_r $$
or
$$ |v^{(n)}_{ne_r-i}| \leq c^{i-1}\ \ {\rm for} \ \ i = 0, \dots, ne_r-1. $$
Therefore
$$ \Vert v_{C_{n,N}} \Vert^2_2 = \sum_{r=1}^{h}  \sum_{i=N}^{n e_r - N} |v^{(n)}_i|^2 \leq \sum_{r=1}^{h} \sum_{i=N}^{n e_r - N} (c^2)^{i-1} \leq h\frac{c^{2(N-1)}}{1 - c^2} < \epsilon^2$$
for all $n$, as required.
\end{proof}

\newpage
{\bf Acknowledgements} The first author should like to thank M Levitin for a number of helpful suggestions.

Department of Mathematics\\
King's College\\
Strand\\
London WC2R 2LS \\ \\
E.Brian.Davies@kcl.ac.uk \\
Paul.Incani@kcl.ac.uk

\end{document}